\documentclass[11pt]{amsart}
\usepackage[utf8]{inputenc}
\usepackage[normalem]{ulem}

\usepackage{enumitem}
\usepackage{mathtools}
\usepackage{epsfig,xcolor,mathrsfs,amsbsy,amsfonts,amsmath,amssymb,amsthm,fullpage, hyperref,cancel, colonequals,tikz,tikz-cd}
\usepackage{cleveref}
\usepackage{relsize}
\usepackage{comment}

\newcounter{thm}
\newtheorem{theorem}[thm]{Theorem}
\numberwithin{thm}{section}
\newtheorem{definition}{Definition}

\newtheorem{prop}[thm]{Proposition}
\newtheorem{cor}[thm]{Corollary}
\newtheorem{lemma}[thm]{Lemma}
\newtheorem{algorithm}[thm]{Algorithm}

\newtheorem{remark}{Remark}
\newtheorem{example}{Example}
\newtheorem*{example*}{Example}

\def \Frob{\text{Frob}}

\def\F{\mathbb F}
\def\C{\mathbb C}
\def\H{\mathbb H}
\def\O{\mathcal O}

\def\Q{\mathbb Q}

\def\Im{\mathrm{Im\,}}
\def\Z{\mathbb Z}

\def \l {\lambda}

\def\({\left(}
\def\){\right)}

\def\G{\Gamma}

\newcommand*\HYPERskip{&}
\catcode`,\active
\newcommand*\pFq{
	\begingroup
	\catcode`\,\active
	\def ,{\HYPERskip}%
	\doHyper
}
\catcode`\,12
\def\doHyper#1#2#3#4#5{%
	\, _{#1}F_{#2}\left[\begin{matrix}#3 \smallskip \\  #4\end{matrix} \; ; \; #5\right]%
	\endgroup
}

\catcode`,\active

\catcode`\,12

\catcode`,\active

\catcode`\,12

\catcode`,\active

\catcode`\,12

\catcode`,\active

\catcode`\,12

\def \Tr {\text{Tr}}

\DeclareMathOperator{\SL}{SL}

\DeclareMathOperator{\ord}{ord}
\DeclareMathOperator{\univ}{univ}
\DeclareMathOperator{\Gal}{Gal}
\DeclareMathOperator{\Leg}{Leg}

\DeclareMathOperator{\Spec}{Spec}
\DeclareMathOperator{\Tate}{Tate}
\DeclareMathOperator{\Sym}{Sym}
\DeclareMathOperator{\coker}{coker}
\DeclareMathOperator{\Res}{Res}
\DeclareMathOperator{\id}{id}
\DeclareMathOperator{\Hdg}{Hdg}
\DeclareMathOperator{\an}{an}
\DeclareMathOperator{\End}{End}

\let\originalleft\left
\let\originalright\right
\renewcommand{\left}{\mathopen{}\mathclose\bgroup\originalleft}
\renewcommand{\right}{\aftergroup\egroup\originalright}

\title{Atkin and Swinnerton-Dyer congruences for meromorphic modular forms}


\address{Department of Mathematics and Computer Science, Wesleyan University, Middletown, CT 06459}
\email{mallen02@wesleyan.edu}

\address{Department of Mathematics, Louisiana State University, Baton Rouge, LA 70803}
\email{llong@lsu.edu}

\address{Department of Mathematics, Louisiana State University, Baton Rouge, LA 70803}
\email{hsaad@lsu.edu}

\author{Michael Allen, Ling Long, and Hasan Saad}
\date{}

\begin{document}

\begin{abstract}
In the 1970's, Atkin and Swinnerton-Dyer conjectured that Fourier coefficients of holomorphic modular cusp forms on noncongruence subgroups of $\SL_2(\Z)$ satisfy certain $p$-adic recurrence relations which are analogous to Hecke's recurrence relations for congrunece subgroups. In 1985, this was proven in seminal work of Scholl and it was recently extended to weakly holomorphic modular forms by Kazalicki and Scholl. We show that Atkin and Swinnerton-Dyer type congruences extend to the setting of meromorphic modular forms and that the $p$-adic recurrence relations arise from Scholl's congruences in addition to a contribution of fibers of universal elliptic curves at the poles. Moreover, when the poles are located at CM points, we exploit the CM structure to reduce these $p$-adic recurrence relations to $2$-term relations and we give explicit examples. Using this framework, we partially prove conjectures that certain meromorphic modular forms are magnetic.
\end{abstract}

\maketitle
\tableofcontents

\section{Introduction and statement of results}\label{sec:intro}

In the early 1900's, Ramanujan conjectured and Mordell proved \cite{MordellRamanujan} that if we consider the cusp form of weight $12$ on $\SL_2(\Z)$
$$
\Delta(\tau) := q\prod\limits_{n=1}^\infty(1-q^n)^{24}=:\sum\limits_{n=1}^\infty \tau(n)q^n,
$$
then for all primes $p$ and $n\geq 1,$ we have
\begin{equation}
\tau(pn)=\tau(n)\tau(p)-p^{11}\cdot\tau(n/p).
\end{equation}
This was later generalized by Hecke who built the theory of Hecke operators for congruence subgroups of $\SL_2(\Z).$ Using this theory, Hecke showed that the space $S_k(\SL_2(\Z))$ of weight $k$ cusp forms decomposes into simultaneous Hecke eigenforms $\displaystyle f(\tau)=\sum_{n \geq 0} a(n)q^n$
for which for all $n,s\geq 1$ and all primes $p,$ we have
\begin{equation}\label{eq:Hecke-recurrence}
    a(np^{s}) = a(p)a(np^{s-1}) - p^{k-1} a(np^{s-2}).
\end{equation} 

When the assumptions above are loosened---in particular, if we consider noncongruence subgroups of $\SL_2(\Z)$ or meromorphic modular forms---the resulting spaces of modular forms are no longer equipped with a nice Hecke theory and such recurrences no longer hold.  However, Atkin and Swinnerton-Dyer observed \cite{ASD} that $p$-adic analogues of \eqref{eq:Hecke-recurrence} hold for non-congruence subgroups. To make this precise, suppose that $\Gamma$ is a finite-index subgroup of $\SL_2(\Z)$ defined over $\Q$ (see 5.1 of \cite{Scholl2}) and denote by $S_k(\Gamma)$ the space of weight $k$ cusp forms on $\Gamma.$ If $\dim S_k(\Gamma)=1$ and $\mu$ denotes the cusp width of $i\infty$ in $\Gamma,$ choose a nonzero form $f(\tau)$ on $\Gamma$ such that
$$
f(\tau):=\sum\limits_{n=1}^\infty a(n)e^{\frac{2\pi in\tau}{\mu}}\in S_k(\Gamma)
$$
with $k^n\cdot a(n)\in\Z\left[\frac{1}{M}\right]$ for some $M\geq 1$ and $k\in\C$ satisfying $k^\mu\in\Z\left[\frac{1}{M}\right]^\ast.$ In this notation, for almost all primes $p,$ there exists an integer $A_p$ with $|A_p|\leq 2p^{\frac{k-1}{2}}$ and a Dirichlet character $\chi$ such that if $m,s\geq 1,$ then we have (see Theorem in Section 1 of \cite{Scholl2})
\begin{equation}\label{eq:ASD-congruence-dim1}
    a(mp^s) \equiv A_p a(mp^{s-1}) -\chi(p) p^{k-1} a(mp^{s-2}) \mod{p^{(k-1)s}}.
\end{equation}
In his seminal work, Scholl proved \cite{Scholl2} that an analogue of \eqref{eq:ASD-congruence-dim1} holds for any $d=\dim S_k(\Gamma)$. Moreover, in recent work, Kazalicki and Scholl extended \cite{KS16} this result for weakly holomorphic modular forms, that is, modular forms which are holomorphic on the upper half-plane with possible poles at the cusps. 

More recently, there has been growing interest in the arithmetic properties of Fourier coefficients of modular forms when the poles lie at CM points on the upper half-plane (for example, see  \cite{Yingkun}, \cite{PasolZudilin},\cite{Bonisch}, \cite{Bogo-Li-Schwagenscheidt},\cite{Pengcheng}) and the interest in these forms originated in part due to work of Broadhurst and Zudilin \cite{Broadhurst-Zudilin-magnetic}.  To give an example of such properties, we first set some notation. If $\tau\in\mathfrak{H}$, the upper-half of the complex plane,  and $q:=e^{2\pi i\tau},$ let  $\phi_0(\tau)$ be the unique normalized Eisenstein series of weight $2$ on $\Gamma_0(2),$  $A(\tau)$ a Hauptmodul on $\Gamma_0(2)$, and $E_4(\tau)$ the normalized Eisenstein series of weigh $4$ on $\SL_2(\Z),$ given by the following $q$-expansions:

\begin{align*}
    \phi_0(\tau)&:=1+24\sum\limits_{m=1}^\infty\frac{mq^m}{1+q^m}, \\
    A(\tau)&:=q\prod\limits_{m=1}^\infty (1+q^m)^{24}, \\
    E_4(\tau)&:= 1+240\sum\limits_{m=1}^\infty\left(\sum\limits_{d\mid m} d^3\right)q^m.
\end{align*}

Moreover, consider the meromorphic modular form $C_4(\tau)$ of weight $4$ on $\Gamma_0(2)$ given by
\begin{equation}\label{eq:C4}
    C_4(\tau) := \frac{A(\tau)}{3(1+64A(\tau))}(4\phi_0^2(\tau)-E_4(\tau))=:\sum\limits_{n=1}^\infty c(n)q^n.
\end{equation}
In this notation, B\"onisch, Duhr, and Maggio \cite[Section B.1]{Bonisch} conjecture that the coefficients $\frac{c(n)}{n}$ have globally bounded denominators. As a consequence of our main theorem, we show that for all primes $p\geq 3$ and $m,s\geq 1,$ we have (see Subsection~\ref{subsec:Bonsich}) 
$$
c(mp^s)\equiv p\cdot\left(\frac{-1}{p}\right)c(mp^{s-1})\pmod{p^{3s}},
$$
where $\left(\frac{-1}{p}\right)$ is the Legendre symbol. In particular, this yields the desired globally bound denominators property for all odd $n$.

To do this, we extend the work of Kazalicki and Scholl to include meromorphic modular forms. To make this precise, we first introduce some notation (see \cite{Sch85b} for this setup).

In what follows, we denote by $K$ an algebraic number field and by $\mathfrak{o}_K$ its ring of integers.  If $N\geq 3,$ consider the modular curve $X(N)$ over $\mathfrak{o}_K,$ let $Y(N)$ be the open subset of $X(N)$ parameterizing non-degenerate elliptic curves and let $Z(N)\subset X(N)$ be the cuspidal subscheme. At the cusp $\underline{\infty}',$ we denote the formal uniformizing parameter by $q^{1/N}.$

Now, let $\Gamma$ be a finite-index subgroup of $\SL_2(\Z)$ and let $X$ be a realization of $\Gamma\backslash\overline{\mathfrak{H}},$ where $\overline{\mathfrak{H}}$ is $\mathfrak{H}$ with the cusps adjoined. We assume that $X$ is defined over $R:=\mathfrak{o}_K\left[\frac{1}{M}\right]$ for some $M\geq 1, N\mid M$ and that $X$ is equipped with a finite morphism $\kappa:X\to X(N)\left[\frac{1}{M}\right]$\footnote{In this paper, $\kappa$ always denotes this map.} such that $\kappa$ is \'etale over $Y(N)$ and such that there is a section $\underline{\infty}:\Z[\frac{1}{M}]\to X$ lying above the cusp $\underline{\infty}'$ of $X(N).$ We denote by $Y$ and $Z$ the reduced inverse images of $Y(N)$ and $Z(N)$ respectively. We have that 
$$
\widehat{\O_{X,\underline{\infty}}}=R[[t]],
$$
where there exists $\nu\geq 1$ and $D\in R^\ast$ such that $D\cdot t^\nu=q.$ Moreover, let $p\nmid 2M$ be a prime such that $p$-adic completion $K_p$ of $K$ is an unramified extension of $\Q_p.$ If $\mathfrak{o}_{K_p}$ denotes the ring of integers of $K_p,$ then by Hensel's lemma there exists a unique 
\begin{equation}\label{eq:gammapDef}
\gamma_p\in 1+p\mathfrak{o}_{K_p}\ \ \ \text{ such that }\ \ \ \gamma_p^\nu=\frac{D^p}{\sigma(D)},
\end{equation}
where $\sigma$ is the Frobenius endomorphism of $K_p.$
\begin{remark}
    We give the following remarks regarding the above setup.
    \begin{enumerate}
    \item When $X=X(N),$ we have that $D=\gamma_p=1$ and $\nu=N.$
    \item By a theorem of B\'elyi \cite{Belyi}, if $X/K$ is any smooth, projective, and irreducible curve, then it can be realized as $\Gamma\backslash\overline{\mathfrak{H}}$ for some finite-index subgroup $\Gamma$ of $\SL_2(\Z),$ where $\overline{\mathfrak{H}}$ is $\mathfrak{H}$ with the cusps adjoined. 
    \end{enumerate}
\end{remark}

In this notation, if $k\geq 3$ and $\mathfrak{u} = \{u_1,\ldots, u_t\}\subset Y(K),$ we define the space of $R$-valued meromorphic modular forms with poles in $\mathfrak{u}$ (see Subsection~\ref{subsec:alg-theory})
\begin{align*}
M_k(X,\star\mathfrak{u}, R) :=&\{f\text{ is a meromorphic modular form on $X$ of weight $k,$} \\ 
& \text{holomorphic on $X\setminus\mathfrak{u}$,}\text{ with Fourier coefficients in }R\}
\end{align*}
and the space of $R$-valued meromorphic cusp forms with poles in $\mathfrak{u}$
$$
S_k(X,\star\mathfrak{u},R):=\{f\in M_k(X,\star\mathfrak{u},R) |  f\text{ vanishes at the cusps}\}.
$$
If $\mathfrak{u}$ is empty, we denote the corresponding spaces by $M_k(X,R)$ and $S_k(X,R)$ respectively. Moreover, if $\mathfrak{u}$ is instead a subset of the cusps, we denote the space of weakly holomorphic modular forms by $M_k^{!}(X,R)$ and the subspace with vanishing constant terms in the $q$-expansions by $S_k^{!}(X,R).$

If $f\in S_k(X,\star\mathfrak{u}, R),$  it has a $t$-expansion $\sum\limits_{n\geq 1}a_f(n)t^n$ and this $t$-expansion determines $f$ uniquely (see 1.6.1 of \cite{Katz-padic}). Finally, if $u\in\mathfrak{u}$ and $f\in S_k(X,\star\mathfrak{u},R),$ we denote by $-\ord_u(f)$ the order of the pole of $f$ at $u$ and we write $-\ord_{\mathfrak{u}}(f):=\max\{0,-\ord_u(f) | u\in\mathfrak{u}\}.$ 

Moreover, to make our statements cleaner, we give the following definition.

\begin{definition}\label{def:goodPrime}
    If $M\geq 1, k\geq 3,$ $K$ is an algebraic number field, and $\mathfrak{u}\subset Y(K),$ we say that $p$ is a good prime for $(M,k,K,\mathfrak{u})$ if the following are true.
    \begin{enumerate}
        \item The $p$-adic completion $K_p$ of $K$ is an unramified extension of $\Q_p$ of degree $r.$
        \item The points $u\in\mathfrak{u}$ can be embedded into $\mathfrak{o}_{K_p}.$
        \item The reductions of $u\in\mathfrak{u}$ in $\F_{p^r}$ are distinct.
        \item $p\nmid 2M\cdot(k-2)!.$
    \end{enumerate}
\end{definition}

In this notation, we have an analogue of \eqref{eq:ASD-congruence-dim1} for meromorphic modular forms.

\begin{theorem}\label{thm:main-ASD}
    Let $\Gamma$ be a finite index subgroup of $\SL_2(\Z)$ such that $X_\Gamma:=\Gamma\backslash\overline{\mathfrak{H}}$ is defined over $K$ and satisfies the conditions above. Moreover, let $k\geq 3,$ $\mathfrak{u}\subset Y(K),$ and $d:=2\dim S_k(X_{\Gamma})+(k-1)\cdot\#\mathfrak{u}.$ If $p$ is a good prime for $(M,k,K,\mathfrak{u})$ and $[K_p:\Q_p]=r,$ then there exists a polynomial 
    $$
    P_{k,p^r,\mathfrak{u}}(T)=\sum\limits_{i=0}^{d} A_{i,p^r,\mathfrak{u}} T^i \in \Z[T]
    $$
    such that for any $f(\tau)=\sum\limits_{n\geq 1} a_f(n)t^n\in S_k(X_\Gamma,\star\mathfrak{u}, \mathfrak{o}_K\left[\frac{1}{M}\right])$ and any $m\geq 1, s\geq 1,$ we have
    \begin{equation}
    \sum\limits_{i=0}^d p^{r(k-1)i}\cdot A_{i,p^r,\mathfrak{u}}\cdot\gamma_p^{mp^s(p^{ri}-1)/(p-1)}\cdot a_{f}\left(\frac{mp^s}{p^{ri}}\right) \equiv 0\pmod{ p^{r(k-1)s+j_{f,p}}},
    \end{equation}
    where
    $$
    j_{f,p}:=\begin{cases}
        k+\ord_{\mathfrak{u}}(f) &\text{ if }\dim S_k(X) =0 \text{ and }-\ord_{\mathfrak{u}}(f)\leq k-1,\\
        \ord_p((-\ord_{\mathfrak{u}}(f)-1)!)&\text{ otherwise.}
    \end{cases}
    $$
\end{theorem}

In Lemma~\ref{lem:ladicChar}, we give a formula for the polynomial $P_{k,p^r,\mathfrak{u}}(T)$ in terms of an $\ell$-adic representation. We illustrate Theorem~\ref{thm:main-ASD} by two examples.

\begin{example}\label{ex:ex1}
    Consider the Hauptmodul $\lambda(\tau)$ of $\Gamma_1(4)$ defined by
    \begin{equation}\label{eq:defLambda}
    \lambda(\tau):=16 q\cdot\frac{\prod\limits_{n=1}^{\infty} (1+q^{2n})^{16}}{\prod\limits_{n=1}^{\infty} (1+q^{n})^8}
    \end{equation}
    and the usual theta function given by
    \begin{equation}\label{eq:defTheta}
    \theta(\tau):=\sum\limits_{n\in\Z} q^{n^2}.
    \end{equation}
    If $u\in\overline{\Q}{ \setminus\{0,1\}},$ define the  modular form 
    $$
    f_u(\tau):=\frac{1}{2\pi i}\cdot \frac{\theta(\tau)^2\cdot\frac{d\lambda(\tau)}{d\tau}}{\lambda(\tau)-u} =:\sum\limits_{n\geq 1}a_f(n; u)q^n\in S_3\left(X_1(4), \star \{\lambda(\tau) = u\}, \Z\left[\frac{1}{6u}\right]\right).
    $$ 
    If $p\geq 5$ with $u\in\Z_p$ and $m,s\geq 1,$ then we have
    \begin{equation}\label{eq:ex1prelim}
        A_{0,p}(u)a_f(mp^s; u)+p^2A_{1,p}(u)a_f(mp^{s-1};u)+p^4A_{2,p}(u)a_f(mp^{s-2};u)\equiv 0\pmod{p^{2s+2}}
    \end{equation}
    for some $A_{0,p}(u),A_{1,p}(u),A_{2,p}(u)\in\Z_p.$ We know (see Lemma~\ref{lem:ladicChar} and Theorem 1 of \cite{HLLT}) that $A_{0,p}(u)=p^3,A_{2,p}(u)=1,$ and $A_{1,p}(u)=-pA_p(u),$ where 
    $$
    A_p(u) := p + 1 - \#E_u(\F_p),
    $$
    with $E_u$ the elliptic curve in Theorem~\ref{thm:Gamma1(4)Basis}.
    Dividing (\ref{eq:ex1prelim}) by $p^3,$ this is equivalent to
    \begin{equation}\label{eq:ex1}
    a_f(mp^s;u) - A_p(u)a_f(mp^{s-1};u) + pa_f(mp^{s-2};u) \equiv 0 \pmod {p^{2s-1}}.
    \end{equation}
    \end{example}

\begin{remark}
    To be precise, Theorem~\ref{thm:main-ASD} does not immediately apply to Example~\ref{ex:ex1} since we require $X_1(N)$ with $N\geq 5.$ In Subsection~\ref{subsec:ERCGamma1(4)}, we illustrate how to obtain (\ref{eq:ex1}) from the statement of the theorem by a standard trick.
\end{remark}

\begin{example}\label{ex:ex2}
To illustrate Theorem~\ref{thm:main-ASD} for a noncongruence subgroup of $\SL_2(\Z),$ we first introduce some notation. Let
$$
t(\tau):=q\prod\limits_{n=1}^\infty(1-q^n)^{5\left(\frac n5\right)}, 
$$ where $\left(\frac{\cdot}{5}\right)$ denotes the Kronecker symbol.
This is the Hauptmodul of $\Gamma_1(5)$ which  has a pole at the cusp $0$ and vanishes at the cusp $\infty.$ We have that 
$$
F(\tau):=\frac{1}{2\pi i}\cdot\frac{d{t(\tau)}} 
{d\tau}\cdot \sum_{n\ge 0} \left(\sum_{k=0}^n \binom{n}{k}^2\binom{n+k}{k}\right)t(\tau)^n\in M_3\left(\G_1(5), \Z\left[\frac{1}{5}\right]\right)
$$ is a weight-3 Eisenstein series for $\G_1(5)$ which vanishes at all cusps of $\G_1(5)$ except at infinity.
Let
$t_2(\tau)=t(\tau)^{1/2}$ be the Hauptmodul of a genus-$0$ index-2 subgroup $\Gamma_2$ of $\Gamma_1(5)$ defined in \cite{ALL}.  
If $u\in\Q\setminus\{0\},$ consider the meromorphic modular form on $\Gamma_2$
$$
\frac{t_2(\tau)F(\tau)}{(t_2(\tau)-u)}=:\sum\limits_{n\geq 1}b(n)q^{n/2} \in S_3\left(\Gamma_2, \star\{t_2=u\},\Z\left[\frac{1}{10u}\right]\right).
$$
In this notation, if $p\geq 3$ with $p\neq 5, m,s\geq 1$ and $E_u$ has good reduction modulo $p$,  then we have 
\begin{multline*}
    b(mp^s)-(a_p(u)+\chi(p)B_p)b(mp^{s-1})+(\chi(p) B_pa_p(u)+p+\chi(p) p^2)b(mp^{s-2})\\ -\chi(p) (B_p+a_p(u)p^2)b(mp^{s-3})+\chi(p) p^3b(mp^{s-4})\equiv 0\pmod {p^{2s-3}},
\end{multline*}
where $\chi=\left(\frac{-1}{\cdot}\right),$ $a_p(u)=p+1-\#E_u(\F_p)$ for the elliptic curve
$$
E_{u}:\quad y^2+(1-u^2)xy-u^2y=x^3-u^2x^2,
$$
and
$$
q\prod\limits_{n=1}^\infty(1-q^{4n})^6=:\sum\limits_{n\geq 1}B_nq^n.
$$
For more details, see Subsection~\ref{subsec:Gamma_2} and \cite{ALL} by Atkin, Li, and the second author.
\end{example}

The proof of Theorem~\ref{thm:main-ASD} shows that, under somewhat general conditions, the $(d+1)$-term recurrence relation in Theorem~\ref{thm:main-ASD} decomposes into $2$-term recurrence relations. To make this precise, we first prove that for all primes $p\nmid M\cdot(k-2)!,$ we have
$$
\dim_{K}\frac{S_k(X,\star \mathfrak{u},K)}{\partial^{k-1}M_{2-k}(X,\star \mathfrak{u}, K)}=2\dim_{K}S_k(X)+(k-1)\cdot\#\mathfrak{u},
$$
where $\partial$ is the differential operator whose effect on the $q$-expansions is given by $q\frac{d}{dq}.$ For each prime $p$ and $r$ as in Theorem~\ref{thm:main-ASD} such that $P_{k,p^r,\mathfrak{u}}(T)$ has distinct roots, the space $\frac{S_k(X,\star \mathfrak{u},K)}{\partial^{k-1}M_{2-k}(X,\star \mathfrak{u}, K)}$ has an $\mathfrak{o}_K$-basis $\{f_{i,p^r}=\sum\limits_{n\geq 1}a_{i,p^r}(n)t^n, 0\leq i\leq d-1\}$ such that for all $m\geq 1,s\geq 0,$ we have 
$$
p^{r(k-1)}\gamma_p^{mp^s(p^r-1)/(p-1)}a_{i,p^r}\left(\frac{mp^s}{p^r}\right)- r_i\cdot a_{i,p^r}(mp^s)\equiv 0\pmod{p^{r(k-1)s+j_{f_i,p}}},
$$
where $\{r_0,\ldots,r_{d-1}\}$ are the roots of the polynomial $P_{k,p^r, \mathfrak{u}}(T).$ When $\dim S_k(X)=0$ and the points $u\in\mathfrak{u}$ are preimages under $\kappa:X\to X(N)\left[\frac{1}{M}\right]$ of CM points for the universal elliptic curve $E^{\univ}$ over $X(N),$ the choice of $f_{i,p^r}$ is independent of $p^r.$

\begin{cor}\label{cor:ASD-basis}
Let $K$ be an algebraic number field, $k\geq 3$ such that $\dim S_k(X)=0$ and $\mathfrak{u}\subset Y(K)$ such that $E^{\univ}_{\kappa(u)}$ is a CM elliptic curve by a quadratic order in $K$ for each $u\in\mathfrak{u}.$ There exist modular forms $\{f_{u,i}=\sum\limits_{n\geq 1}a_{u,i}(n)t^n\in S_k(X,\star u,\mathfrak{o}_K\left[\frac{1}{M},u,\frac{1}{u}\right])\footnote{With this notation, we mean adjoining the ring of definition of $u$ and $\frac{1}{u}$ to $\mathfrak{o}_K\left[\frac{1}{M}\right]$}|u\in\mathfrak{u},0\leq i\leq k-2\}$ such that for all primes $p$ good for $(M,k,K,\mathfrak{u})$ with $E^{\univ}_{\kappa(u)}$ having good reduction, we have
$$
\frac{S_k(X,\star u,K)}{\partial^{k-1}M_{2-k}(X,\star u,K)}\cong \bigoplus\limits_{u\in\mathfrak{u}}\bigoplus\limits_{i=0}^{k-2}K\cdot f_{u,i}
$$
satisfying
$$
p^{k-1}\gamma_p^{mp^s}\sigma(a_{u,i}(mp^{s-1})) \equiv p\cdot c_{u,i} \cdot a_{u,\tilde{\sigma}(i)}(mp^s) \pmod{p^{(k-1)s+j_{f,p}}} \ \ \ \ \ \text{ for all }m,s\geq 1,
$$
where $\sigma$ is the Frobenius endomorphism of $K_p,$ $c_{u,i}\in p^{k-2-i}\mathfrak{o}_{K_p},$ and $\tilde{\sigma}(i):=i$ if $\sigma$ fixes the CM field in $K_p$ and $\tilde{\sigma}(i) := k-2 - i$ otherwise. If $[K_p:\Q_p]=r$ and $0\leq i\leq\frac{k-2}{2},$ then we have that
$$
d_{u,i}^r +\frac{p^r}{d_{u,i}^r}=p^{ri}\cdot\left(p^{r(k-2-2i)}+1-\#E(\F_{p^{r(k-2-2i)}})\right),
$$
where 
$$
d_{u,i}:=\sigma^{r-1}\left(c_{u,i}^{\lfloor(r+1)/2\rfloor}\right)\cdot\sigma^r(c_{u,\tilde{\sigma}(i)}^{\lfloor r/2\rfloor}).
$$
Moreover, if $E^{\univ}_{\kappa(u)}$ has ordinary reduction then $\ord_p(c_{u,i})=p^{k-2-i}.$
\end{cor}

The basis in Corollary~\ref{cor:ASD-basis} can be computed explicitly for any choice of CM-poles at once. To illustrate the method, we compute the basis for $\SL_2(\Z)$ and $\Gamma_1(4)$ when we have one pole. 

For $\Gamma_1(4),$  we have the following result.
\begin{theorem}\label{thm:Gamma1(4)Basis}

If $K$ is an algebraic number field and $u\in K\setminus\{0,1\},$ consider the elliptic curve
$$
E_u:\ \ \ y^2=4x^3-\frac{(1+14u+u^2)}{12}x+\frac{1-33u-33u^2+u^3}{216}
$$
and denote by $\omega=\frac{dx}{y}$ and $\eta=\frac{xdx}{y}$ the differential forms on $E_u$ with poles at the origin of order $0$ and $2$ respectively. Assume that $E_u$ is a CM elliptic curve by an order $\mathcal{O}$ of  an imaginary subfield of $K$, $p\geq 5$ is a prime that splits in $K,$ and $c_1\omega+c_2\eta$ is an eigenvector of CM with $c_1,c_2\in K$ and $c_2\neq 0.$ Moreover, suppose that $E_u$ has good and ordinary reduction at $p$ and denote by $\mu_p$ the $p$-adic unit that is a solution of 
$$
\mu_p+\frac{p}{\mu_p}=p+1-\#E_u(\F_p).
$$
In this notation, the following are true.
\begin{enumerate}
    \item If $m,s\geq 1$ and 
    $$
    f(\tau):=\frac{1}{2\pi i}\frac{\theta^2(\tau)\lambda'(\tau)}{\lambda(\tau)-u}=:\sum\limits_{n\geq 1}a_{f}(n)q^n,
    $$
    then we have
    $$
    a_f(mp^s)\equiv \mu_pa_f(mp^{s-1})\pmod{p^{2s}}.
    $$
    \item  If $m,s\geq 1$ and
    \begin{align*}
    g(\tau):=\frac{1}{2\pi i}\Bigg(\frac{c_2\cdot u(1-u)}{(\lambda(\tau)-u)^2} 
    &+\left(c_1-c_2\cdot\frac{(5u-1)}{12}\right)\cdot\frac{1}{(\lambda(\tau)-u)}\Bigg)\theta^2(\tau)\lambda'(\tau) =:\sum\limits_{n\geq 1}a_g(n)q^n,
    \end{align*}
    then we have
    $$
    a_g(mp^s)\equiv \frac{p}{\mu_p}a_g(mp^{s-1})\pmod{p^{2s}}.
    $$

    \item If $m,s\geq 1$ and 
    $$
    f(\tau):=\frac{1}{2\pi i}\frac{\theta^4(\tau)\lambda'(\tau)}{\lambda(\tau)-u}=:\sum\limits_{n\geq 1}a_f(n)q^n,
    $$
    then we have
    $$
    a_f(mp^s)\equiv \mu_p^2 a_f(mp^{s-1})\pmod{p^{3s}}.
    $$

    \item If $m,s\geq 1$ and 
    \begin{align*}
    g(\tau):=\frac{1}{2\pi i}\Bigg(\frac{c_2u(1-u)}{2(\lambda(\tau)-u)^2}+\left(c_1-c_2\cdot \frac{(5u-1)}{12}\right)\frac{1}{(\lambda(\tau)-u)}\Bigg)\theta^4(\tau)\lambda'(\tau)=:\sum\limits_{n\geq 1}a_g(n)q^n,
    \end{align*}
    then we have
    $$
    a_g(mp^s)\equiv pa_g(mp^{s-1})\pmod{p^{3s}}.
    $$

    \item  If $m,s\geq 1$ and
    \begin{align*}
        h(\tau):=\frac{1}{2\pi i}&\Bigg(\frac{c_2^2u^2(1-u)^2}{(\lambda(\tau)-u)^3}+\frac{c_2^2Q(u)+c_1c_2u(1-u)}{(\lambda(\tau)-u)^2} \\
        &+\left(c_2^2P(u)+c_1^2-\frac{(5u-1)c_1c_2}{6}\right)\cdot\frac{1}{(\lambda(\tau)-u)}\Bigg)\theta^4(\tau)\lambda'(\tau)=:\sum\limits_{n\geq 1}a_h(n)q^n,
    \end{align*}
    with $P(u)=\frac{1}{144}(61u^2-46u+1)$ and $\displaystyle Q(u)=\frac{-u(1-u)(17u-7)}{12},$
    then we have
    $$
    a_h(mp^s)\equiv \frac{p^2}{\mu_p^2}a_h(mp^{s-1})\pmod{p^{3s}}.
    $$
    
\end{enumerate}
\end{theorem}

\begin{remark}
    For the appearance of this family of elliptic curves $E_u,$ see Lemmas~\ref{lem:FBResidue} and \ref{lem:E4E6Gamma1(4)}.
    See Subsection~\ref{ss:Diagonalizing-CM-EC} for an algorithm to compute $c_1$ and $c_2$ for a given CM elliptic curve $E_u$.  Note that $f,g,$ and $h$ are unique up to scalar multiples. Moreover, their expressions can be computed using different models of the elliptic curve and using different methods. For example, see Remark~\ref{rem:GMHyper}.
\end{remark}

\begin{remark}
    If $p$ does not split in $K,$ then similar statements hold as in Corollary~\ref{cor:ASD-basis} for the same choices of $f,g,$ and $h.$
\end{remark}

\begin{example}\label{ex:ex3}
If $u=2$, $E_2$ has  CM by $\Z[2i]$ and we have $c_1=1$ and $c_2=4$ (see Algorithm~\ref{algorithm:H1dRDiag}). In the notation of the previous theorem, up to multiplicative constants, we have
$$
f(\tau)=\frac{1}{2\pi i}\cdot \frac{\theta(\tau)^2\cdot\frac{d\lambda(\tau)}{d\tau}}{\lambda(\tau)-2} =\sum\limits_{n\geq 1}a_f(n)q^n\in S_3\left(X_1(4), \star \{\lambda = 2\}, \Z\left[\frac{1}{2}\right]\right)
$$ 
and
$$
g(\tau)=f(\tau)\cdot\frac{\lambda(\tau)+2}{\lambda(\tau)-2}=\sum\limits_{n\geq 1}a_g(n)q^n.
$$
For each prime 
$p\equiv 1{\pmod 4},$ let $\mu_{2,p}$ be the unit root solution to $\mu_{2,p}+\frac{p}{\mu_{2,p}}=p+1- \#E_2(\F_p).$ 
In this notation, Theorem~\ref{thm:Gamma1(4)Basis} states that for all $m,s\geq 1,$ we have
\begin{equation*}
    a_f(mp^{s})\equiv\mu_{2,p} a_f(mp^{s-1})\pmod{p^{2s}} \ \ \ \text{ and }\ \ \ a_g(mp^{s})\equiv\frac{p}{\mu_{2,p}}a_g(mp^{s-1})\pmod{p^{2s}}.
\end{equation*}
Moreover, when $p\equiv 1\pmod 4,$ we can write 
$$
\mu_{2,p}=-\left(\frac2p\right)\G_p(\frac12)\G_p(\frac14)^2,
$$
where $\Gamma_p$ is the $p$-adic Gamma function (see (4.4) of \cite{Li-Long14} and Theorem 1 of \cite{LR}). Motivated by \cite[Theorem 3]{LR}, we find that if $p\equiv 3\pmod 4$ and $mp^s\leq 6000,$ then we have that
$$
a_g(mp^s)\equiv \frac{4}{\mu_p}a_f({m}p^{s-1}) \pmod{p^{2s}}\ \ \ \text{ and }\ \ \ a_f(mp^s)\equiv -\frac{p\mu_p}{4}a_g({m}p^{s-1}) \pmod{p^{2s}}.
$$
Although these two constants satisfy the relations in Corollary~\ref{cor:ASD-basis}, the authors are not aware of their geometric meaning when $p\equiv 3\pmod 4$.

\end{example}

For the group $\SL_2(\Z),$ we partially prove a recent conjecture of Zhang who constructs the basis using analytic methods. To make this precise, we first introduce some notation. If $k\geq 4,$ denote by $E_k(\tau)$ the normalized Eisenstein series of weight $k,$ by $\Delta(\tau)$ the unique normalized cusp form of weight $12$ and level $1,$ and by $j(\tau)$ the $j$-invariant. Moreover, if $k\in\Z$ and $f(\tau)$ is a function on $\mathfrak{H},$ we make use of the raising operator
\begin{equation}\label{eq:R(k,tau)}
    R_{k,\tau}f(\tau):=\frac{1}{2\pi i}\frac{\partial f}{\partial\tau}-\frac{k}{4\pi\Im(\tau)}\cdot f(\tau).
\end{equation}

In this notation, we have the following theorem.

\begin{theorem}\cite[Conjecture 4.2]{Pengcheng}\label{thm:SL2(Z)Basis}
Let $a,b,c\in\Z$ with $\gcd(a,b,c)=1, D=b^2-4ac, \alpha_D=\frac{-b+\sqrt{-D}}{2a}  \in\mathfrak{H}$ such that $j(\alpha_D)\neq 0,1728$ and consider the elliptic curve 
$$
E_{j(\alpha_D)}:\ \ \ \ y^2+xy=x^3-\frac{36}{j(\alpha_D)-1728}x-\frac{1}{{j(\alpha_D)-1728}}.
$$ For each prime $p$  which splits in $\Q(\sqrt{-D}, j(\alpha_D))$ such that $E_{\alpha_D}$ has good reduction at $p$, let $\mu_p$ be the $p$-adic unit solution to 
$$
\mu_p+\frac{p}{\mu_p}=p+1-\#E_{j(\alpha_D)}(\F_p).
$$
Moreover, for each $k\in\{4,6,8,10,14\}$ and $0\le r\le k-2,$ let \begin{equation}\label{eq:f(k,r,z)}
    f_{k,r,\alpha_D}(\tau):=E_k(\tau)\cdot R_{2-k,z}^r\left(\frac{E_{14-k}(z)}{\Delta(z)\cdot(j(\tau)-j(z))}\right)\Bigg|_{z=\alpha_D}=:\sum\limits_{n\geq 1} a_{{k},r,\alpha_D}(n)q^n,
\end{equation}
where $R_{2-k,z}^r:=R_{2r-k,z}\circ \cdots \circ R_{4-k,z} \circ R_{2-k,z}.$ 
If $p\nmid (k-2)!$ and $m,s\geq 1,$ then we have
$$
a_{k,r,\alpha_D}(mp^s)\equiv  p^r\mu_p^{k-2-2r}\cdot a_{k,r,\alpha_D}(mp^{s-1}) \pmod{ p^{(k-1)s}}.
$$
\end{theorem}

The proof of Theorem~\ref{thm:main-ASD} relies on the fact that the ring
$\frac{S_k(X,\star u,K)}{\partial^{k-1}M_{2-k}(X,\star u,K)}$ can be described as the first cohomology of a logarithmic de Rham complex. This logarithmic de Rham complex is an extension of the de Rham complex defined by Scholl in \cite{Scholl2} and used by Kazalicki and Scholl in \cite{KS16} for weakly holomorphic modular forms. Using this characterization, one obtains a short exact sequence
$$
\begin{tikzcd}
0 \arrow[r] & {\frac{S_k^{!}(X,K)}{\partial^{k-1}M^{!}_{2-k}(X,K)}} \arrow[r] & {\frac{S_k(X,\star u,K)}{\partial^{k-1}M_{2-k}(X,\star u,K)}} \arrow[r, "\Res"] & \bigoplus\limits_{u\in\mathfrak{u}} \Sym^{k-2}H_{dR}^1(E_u) \arrow[r] & 0,
\end{tikzcd}
$$
where $\Sym^{k-2}H_{dR}^1(E_u)$ is the $(k-2)$-th symmetric power of the de Rham cohomology of the fiber of an elliptic curve lying over the poles and $\Res$ is essentially a residue map.

In order to compute the congruences, we note that, after base-changing to $\Q_p$ (if $p$ splits in $K$), there is a Frobenius operator on the spaces in the short exact sequence, constructed by Scholl, which acts on the $q$-expansions as
$$
F\left(\sum\limits_{n\geq 1}a(n) q^n\right)\equiv p^{k-1}\sum\limits_{n\geq 1}a(n)q^{pn}\mod{\left(q\frac{d}{dq}\right)^{k-1}\Z_p[[q]]}.
$$
The characteristic polynomial of $F$ on $\frac{S_k(X,\star u,\Q_p)}{\partial^{k-1}M_{2-k}(X,\star u,\Q_p)}$ is then computed from the short exact sequence. This gives Theorem~\ref{thm:main-ASD} with an undetermined $j_{f,p}.$ We then apply an integral computation to explicitly determine $j_{f,p}.$

For Corollary~\ref{cor:ASD-basis}, we exploit the global CM action in order to find an ``appropriate'' basis for the action of the Frobenius operator $F.$ Moreover, we use a result of Mazur to determine the behavior of this basis under Frobenius. When the elliptic curve has ordinary reduction, the basis is diagonal and the same result gives the corresponding eigenvalues. In order to determine this basis explicitly, we explicitly determine the meromorphic modular forms whose residues are eigenvectors of the CM action on $\Sym^{k-2}H^1_{dR}(E_u).$

We put all this together to prove Theorem~\ref{thm:main-ASD} and Corollary~\ref{cor:ASD-basis}. Moreover, we use the explicit residue formulas as well as the explicit description of the characteristic polynomial in order to deduce Theorems~\ref{thm:Gamma1(4)Basis} and \ref{thm:SL2(Z)Basis}. 

Finally, we conclude with supplementary discussions on algorithmically diagonalizing de Rham cohomology for an elliptic curve, the examples in the introduction, some present conjectures, and some numerical observations for supersingular primes.

This paper is organized as follows. In Section~\ref{sec:de-Rham-weakly-holomorphic} we recall the necessary constructions and results by Scholl. In Section~\ref{sec:de-Rham-meromorphic}, we extend Scholl's construction to a logarithmic complex and express quotient spaces of meromorphic modular forms as a de Rham cohomology. Doing this, we prove a preliminary version of Theorem~\ref{thm:main-ASD} which has an undetermined shift in the congruence modulus. In Section~\ref{sec:integral-shift-computation}, we perform the necessary integral computations to determine the shift. In Section~\ref{sec:CM-de-Rham}, we prove the necessary lemmas that allow us to deduce Corollary~\ref{cor:ASD-basis} from Theorem~\ref{thm:main-ASD}. In Section~\ref{sec:ERC}, we compute the required residues that allow us to deduce Theorems~\ref{thm:Gamma1(4)Basis} and \ref{thm:SL2(Z)Basis} from the proof of Corollary~\ref{cor:ASD-basis}. In Section~\ref{sec:proofs}, we combine the results of Sections~\ref{sec:de-Rham-meromorphic} through \ref{sec:ExRD} to prove Theorems~\ref{thm:main-ASD}, \ref{thm:Gamma1(4)Basis} and \ref{thm:SL2(Z)Basis} and Corollary~\ref{cor:ASD-basis}. In Section~\ref{sec:ExRD}, we add some supplementary material and discussions.

\section*{Acknowledgment}

The authors are indebted to Jerome W. Hoffman for multiple discussions on the geometry present. The authors are also indebted to Fang-Ting Tu for multiple computations and suggestions which helped in forming the conjectures in this paper. The authors would also like to thank Pengcheng Zhang for sharing his conjectures early on which led to Theorem~\ref{thm:SL2(Z)Basis} as well as introducing the authors to the notion of magnetic modular forms. The second author would also like to thank Frits Beukers for a suggestion for supercongruences which led to her interest in meromorphic modular forms. The authors also thank Wen-Ching Winnie Li, Tong Liu, Ken Ono, and Anthony Scholl for helpful comments and discussions. Long is supported in part by the Simons Foundation grant \#MP-TSM-00002492 and the LSU Michael F. and Roberta Nesbit McDonald Professorship.

\section{de Rham cohomology for weakly holomorphic modular forms}\label{sec:de-Rham-weakly-holomorphic}

In order to prove Theorem~\ref{thm:main-ASD}, given an unramified extension $L/\Q_p$ of degree $r$ with Frobenius morphism $\sigma,$ we describe the quotient $\frac{S_k(X,\star u,L)}{\partial^{k-1}M_{2-k}(X,\star u,L)}$ as a de Rham cohomology group. This cohomology group is finite-dimensional and admits a $\sigma$-linear  Frobenius action $F$ whose effect on $t$-expansions is
$$
F\left(\sum\limits_{n\geq 1} a(n)t^n\right)\equiv p^{k-1}\sum\limits_{n\geq 1}\sigma(a(n))\gamma^{n}t^{pn} \mod{\left(\frac{d}{dt}\right)^{k-1}L[[t]]},
$$
where $\gamma$ is as in (\ref{eq:gammapDef}).  Applying the Cayley--Hamilton theorem to the linear map $F^r$ gives the required congruences. In this section, we recall the necessary background from \cite{Scholl2} that is required for this computation.

\subsection{Algebraic Theory}\label{subsec:alg-theory}

In his proof of Atkin and Swinnerton-Dyer congruences for modular forms on noncongruence subgroups of $\SL_2(\Z),$ Scholl describes the cusp forms on a noncongruence subgroups as a submodule of a de Rham cohomology with nonconstant coefficients. Here we recall the algebraic constructions and results (see Section 2 of \cite{Scholl2} and Section 3 of \cite{KS16}). 

We first describe the situation on $X(N)$ where $N\geq 3.$ We have a group action by $G_N:=\SL(\mu_n\times \Z/N\Z)$ on $X(N)$ such that $\{\pm 1\}$ acts trivially and we have
$$
G_N\backslash X(N)\cong X(1)\left[\frac{1}{N}\right],
$$
where for a curve $X$ and $M\geq 1,$ $X[\frac{1}{M}]$ denotes the base-change of $X$ to $\Spec\Z\left[\frac{1}{M}\right].$ Moreover, the curve $X(N)$ represents a fine moduli problem and we have a generalized universal elliptic curve
$$
\begin{tikzcd}
E^{\univ} \arrow[r, "\pi"] & X(N) \arrow[l, "e", bend left]
\end{tikzcd},
$$
where $e$ is the zero section and $G_N$ acts on $E^{\univ}$ over its action on $X(N).$ If $\Tate(q)/\Z[[q^{1/N}]]$ denotes the Tate curve with $N$ sides, then the canonical level $N$ structure gives us a morphism
$$
\psi:\Spec\Z\left[\frac{1}{N}\right][[q^{1/N}]]\to X(N)
$$
which identifies $\Z\left[\frac{1}{N}\right][[q^{1/N}]]$ with the formal completion of $X(N)$ at the cusp $\underline{\infty}.$ We define $Y(N)$ to be the subscheme that parameterizes actual elliptic curves and we denote by $Z(N)$ the complimentary reduced closed subscheme.

On $Y(N),$ we have the Hodge bundle $\boldsymbol{\omega}:=e^\ast\Omega^1_{E^{\univ}|_{Y(N)}}$ which is a locally free sheaf of $\O_Y$-modules of rank $1$ that admits a canonical extension to $X(N),$ which we also denote by $\boldsymbol{\omega}.$ In this notation, if $R$ is a $\Z\left[\frac{1}{N}\right]$-algebra, the $R$-module of $R$-valued modular forms of weight $k$ on $X(N)$ is defined by 
$$
M_k(X(N),R):=H^0(X(N),\boldsymbol{\omega}^{\otimes k}\otimes R)
$$
and if $k\geq 2,$ the submodule of cusp forms is defined by
$$
S_k(X(N),R):=H^0(X(N),\omega^{\otimes(k-2)}\otimes\Omega_{X(N)}^1\otimes R).
$$
Moreover, if $f$ is an $R$-valued modular form, the  $q$-expansion (at $\underline{\infty}$) is defined as 
$$
\psi^\ast(f)=\left(\sum\limits_{n=0}^\infty a_f(n)q^{\frac{n}{N}}\right)\cdot\omega^k\cdot\frac{dq}{q}\in \Z\left[\frac{1}{N}\right][[q^{1/N}]]\otimes_{\Z} R,
$$
where $\omega$ is a canonical generator for $\psi^\ast\boldsymbol{\omega}.$

The space of cusp forms sits in a short exact sequence that arises from the Hodge filtration of the relative de Rham cohomology of $E^{\univ}$ over $X(N).$ To make this precise, define $\mathcal{E}:=\mathbb{R}^1\pi_\ast\Omega^\bullet_{\pi^{-1}(Y(N))/Y(N)}.$ This sheaf has a canonical extension to a locally free sheaf of $\O_{X(N)}$-modules of rank $2$ and sits in the short exact sequence
$$
\begin{tikzcd}
0 \arrow[r] & \boldsymbol{\omega} \arrow[r] & \mathcal{E} \arrow[r, "\beta"] & \boldsymbol{\omega}^{-1} \arrow[r] & 0
\end{tikzcd}
$$
which ``is'' the Hodge filtration. Moreover, the sheaf $\mathcal{E}$ is equipped with the Gauss--Manin connection
$$
\nabla:\mathcal{E}\to\mathcal{E}\otimes\Omega_{X(N)}^1(\log Z(N))
$$
which has logarithmic poles at the cusps. If $k\geq 3,$ write 
$$
\mathcal{E}_{k-2}:=\Sym^{k-2}\mathcal{E}, \quad \nabla_{k-2}:=\Sym^{k-2}\nabla
$$
and consider the chain-complex
$$
\Omega^\bullet:=[\begin{tikzcd}
\mathcal{E}_{k-2} \arrow[r, "\nabla_{k-2}"] & {\nabla_{k-2}\left(\mathcal{E}_{k-2}\right)+\mathcal{E}_{k-2}\otimes\Omega_{X(N)}^1}
\end{tikzcd}].
$$
In this notation, Scholl proves the following.
\begin{theorem}\cite[Theorem 2.7]{Scholl2}\label{thm:SchollDecomp1}
    If $(k-2)!$ is invertible in $R,$ then the following are true.
    \begin{enumerate}
        \item We have a short exact sequence of $R$-modules
        $$
        \begin{tikzcd}
        0 \arrow[r] & {S_k(X(N),R)} \arrow[r] & {\mathbb{H}^1(X(N),\Omega^{\bullet}\otimes R)} \arrow[r] & {S_k(X(N),R)^\vee} \arrow[r] & 0
        \end{tikzcd},
        $$
        where $\vee$ denotes the $R$-dual and where this is compatible with the action of $G_N.$ Moreover, we have $\mathbb{H}^i(X(N),\Omega^\bullet\otimes R)=0$ for $i\neq 1.$
        \item There is an isomorphism 
        $$
        \begin{tikzcd}
        H^1(\psi^\ast\Omega^\bullet) \arrow[r, "\sim"] & \coker\left(\left(q\frac{d}{dq}\right)^{k-1}:q^{\frac{1}{N}}R[[q^{\frac{1}{N}}]]\to q^{\frac{1}{N}}R[[q^{\frac{1}{N}}]]\right).
        \end{tikzcd}
        $$
        \item (1) and (2) are compatible with base-change.
        \item If $f\in S_k(X(N),R),$ then its class in $H^1(\psi^\ast\Omega^{\bullet}\otimes R)$ is the class of its $q$-expansion.
    \end{enumerate}
\end{theorem}

Scholl extends this setting further to encompass non-congruence subgroups and also to a more general chain complex which gives higher congruence modulus for cusp forms.  The above results generalize almost verbatim if we replace $X(N)$ by a modular curve $X'$ which represents a fine moduli problem, and $G_N$ by an appropriate $G'.$\footnote{If $X'=X(N),$ then $G'$ here is just $G_N.$} Now, let $R$ be a $\Z\left[\frac{1}{N}\right]$-algebra, $X$ a smooth projective curve over $R$ equipped with a finite morphism $\kappa:X\to X'\left[\frac{1}{M}\right]$ where $M\geq 1$ and $N\mid M$ such that $\kappa|_{\kappa^{-1}(Y')}:\kappa^{-1}(Y')\to Y'$ is \'etale, where $Y'$ is the subset of $X'$ parameterizing actual elliptic curves. Moreover, we require that there is a section $\underline{\infty}$ such that $\kappa(\underline{\infty})=\underline{\infty}'$ and an action of a closed flat subgroup scheme $G\subset G'\left[\frac{1}{M}\right]$ that is compatible with the action of $G'$ on $X'.$ We denote by $Y$ and $Z$ the preimages of $Y'$ and $Z'$ respectively, and we abuse notation to denote the pullbacks of $\boldsymbol{\omega},\mathcal{E},\nabla$ by the same symbols.

At the cusps of $X,$ there exist $t$-expansion maps. In particular, the formal completion of $\O_X$ along $\underline{\infty}$ is of the form
$$
A := R\left[\frac{1}{M}\right][[t]],\ \ \ D\cdot t^{\nu} = q^{\frac{1}{N}},
$$
where $\nu\geq 1$ is invertible in $\Z\left[\frac{1}{M}\right]$ and $ D\in\Z\left[\frac{1}{M}\right]^\ast.$ As for $X(N),$ we have an associated morphism $\psi:\Spec A\to X$ (with similar morphisms $\psi_z$ at other cusps $z$) and the $R$-module of $R$-valued modular forms on $X$ of weight $k$ is defined as
$$
M_k(X,R):=H^0(X,\omega^{\otimes k}\otimes R)
$$
with the submodule of cusp forms (for $k\geq 2$) defined as
$$
S_k(X,R):=H^0(X,\omega^{\otimes(k-2)}\otimes\Omega_X^1).
$$
The $t$-expansion map is also defined similarly. If $m\geq 1,$ Scholl defines a generalization of $\mathcal{E}$ as
$$
\mathcal{E}_{1,m}:=\beta^{-1}(m\cdot\boldsymbol{\omega}^{-1})
$$
and
$$
\mathcal{E}_{k-2,m}:=\Sym^{k-2}(\mathcal{E}_{1,m}),
$$
where $\beta:\mathcal{E}\twoheadrightarrow\boldsymbol{\omega}^{-1}.$ Finally, consider the chain complex 
$$
\Omega_m^\bullet:=[\begin{tikzcd}
m\cdot \mathcal{E}_{k-2,m} \arrow[r, "\nabla_{k-2}"] & {\nabla_{k-2}\left(m\cdot \mathcal{E}_{k-2,m}\right)+\mathcal{E}_{k-2,m}\otimes\Omega_{X}^1}
\end{tikzcd}].
$$
In this notation, Scholl proves the following generalization of Theorem~\ref{thm:SchollDecomp1}.

\begin{theorem}\cite[Theorem 2.12]{Scholl2}\label{thm:SchollDecomp2}
    If $R$ is a $\Z\left[\frac{1}{M\cdot (k-2)!}\right]$-algebra and $m\geq 1$ is not a zero-divisor in $R,$ then the following are true.
    \begin{enumerate}
        \item We have a short exact sequence of $R$-modules
        $$
        \begin{tikzcd}
        0 \arrow[r] & {S_k(X,R)} \arrow[r] & {\mathbb{H}^1(X,\Omega_m^{\bullet}\otimes R)} \arrow[r] & {m^{k-1}\cdot S_k(X,R)^\vee} \arrow[r] & 0
        \end{tikzcd},
        $$
        and this is compatible with the action of $G.$ Moreover, we have $\mathbb{H}^i(X,\Omega_m^{\bullet}\otimes R)=0$ for $i\neq 1.$
        \item There is an isomorphism 
        $$
        \begin{tikzcd}
        H^1(\psi^\ast\Omega_m^\bullet\otimes R) \arrow[r, "\sim"] & \coker\left(\left(mq\frac{d}{dq}\right)^{k-1}:R[[t]]\to tR[[t]]\right),
        \end{tikzcd}
        $$
        where $q\frac{d}{dq} = \frac{1}{N\nu}t\frac{d}{dt}.$
        \item (1) and (2) are compatible with base-change.
        \item If $f\in S_k(X,R),$ its class in $H^1(\psi^\ast\Omega_m^{\bullet}\otimes R)$ is the class of its $t$-expansion.
    \end{enumerate}
\end{theorem}

For completeness, we mention that Kazalicki and Scholl describe $\mathbb{H}^1(X,\Omega_m^\bullet\otimes R),$ when $m=1$ and $R$ is a field or a Dedekind domain of characteristic $0,$ in terms of weakly holomorphic modular forms. To make this precise, define the $R$-module of $R$-valued weakly exact cusp forms
\begin{align*}
S_k^{!\text{-ex}}(X,R):=\Bigg\{&f\in M_k^!(X,R) | f\text{ has vanishing constant terms at all the cusps and}\\
&\text{ for every cusp $z$ the principal part of $\psi_z^\ast(f)$ is in the image of $\left(q\frac{d}{dq}\right)^{k-1}$}\Bigg\}.
\end{align*}
In this notation, Kazalicki and Scholl prove the following result.
\begin{theorem}\cite[Theorem 4.3]{KS16}\label{thm:KSChar}
We have an isomorphism
$$
\begin{tikzcd}
{\mathbb{H}^1(X,\Omega_1^\bullet\otimes R)} \arrow[r, "\sim"] & {\frac{S_k^{!\text{-ex}}(X,R)}{\partial^{k-1} M_k^!(X,R)}}.
\end{tikzcd}
$$
\end{theorem}

\subsection{\texorpdfstring{$p$-adic Theory}{p-adic Theory}}\label{subsec:pAdicTheory}
In order to prove the Atkin and Swinnerton--Dyer congruences for weakly holomorphic modular forms and cusp forms on noncongruence subgroups, Scholl shows that the $p$-adic completion of $\mathcal{E}_{k-2}$ and $\nabla$ gives an $F$-crystal structure with logarithmic singularities. This induces a Frobenius action $F$ on the cohomology groups $\mathbb{H}^1(X,\Omega_m^{\bullet}\otimes\Z_p).$ Scholl computes the effect on the Fourier expansion explicitly and also computes the characteristic polynomial of $F$ in terms of the characteristic polynomial of Frobenius on an explicit $\ell$-adic representation. Here we recall the necessary $p$-adic theory (for more details, see Sections 3 and 4 of \cite{Scholl2}).

First, let $N\geq 3$ and let $p\nmid 2N$ be a prime. Let $(\cdot)^\infty$ denote $p$-adic completion. Scholl shows (see Proposition 3.2 of \cite{Scholl2}) that $(\mathcal{E}^\infty,\nabla^\infty)$ is the underlying differential equation of an $F$-crystal with logarithmic singularities on $Z(N)^\infty$ (for the definition of $F$-crystals with logarithmic singularities, see 1.5 of \cite{Scholl1}). Moreover, if $x_0\in X(N)(\F_p),$ $x\in X(N)(\Z_p)$ is a lift of $x_0,$ and $\phi$ is a lift of the absolute Frobenius on $X\otimes\F_p$ with $\phi(x)=x,$ then on the stalk $\mathcal{E}_x^\infty,$ the characteristic polynomial of $F_x$ (this is independent of the choice of $x$ and $\phi$) is given by (see Section 5.1 of \cite{Scholl1} and the construction in Section 3 of \cite{Scholl2})
$$
T^2-a_x(p)T+p,
$$
where $a_x(p):=p+1-\#E^{\univ}_x(\F_p).$ Now, let $X,G,M,\ldots$ be as in Subsection~\ref{subsec:alg-theory}. Using the fundamental theorem of proper morphisms, Scholl proves the following theorem.

\begin{theorem}\cite[Theorem 3.6]{Scholl2}
    Let $k\geq 3,$ $p\nmid M\cdot(k-2)!,$ and $m\in\{1,p\}.$ Moreover, let $L/\Q_p$ be an unramified extension of $\Q_p$ with ring of integers $\mathfrak{o}_L$ and Frobenius morphism $\sigma.$ There is a canonical endomorphism $F$ of $\mathbb{H}^1(X,\Omega^{\bullet}_m\otimes\mathfrak{o}_L)$ which commutes with the action of $G.$ Moreover, if $\rho$ denotes the canonical map (arising from taking $t$-expansions)
    $$
    \rho:\mathbb{H}^1(X,\Omega^{\bullet}_m\otimes\mathfrak{o}_L)\to H^1(\psi^\ast\Omega_{m}^\bullet\otimes\mathfrak{o}_L)\cong\frac{t\mathfrak{o}_L[[t]]}{\left(mq\frac{d}{dq}\right)^{k-1}\mathfrak{o}_L[[t]]}
    $$
    and if $f\in\mathbb{H}^1(X,\Omega_m^{\bullet}\otimes\mathfrak{o}_L)$ with
    $$
    \rho(f)\equiv\sum\limits_{n\geq 1} a_f(n)t^n \mod \left(mq\frac{d}{dq}\right)^{k-1}\mathfrak{o}_L[[t]],
    $$
    then we have
    $$
    \rho(F(f))\equiv \sum\limits_{n\geq 1}p^{k-1}\cdot \sigma(a_f(n))\gamma_p^nt^{np}\mod \left(mq\frac{d}{dq}\right)^{k-1}\mathfrak{o}_L[[t]].
    $$
\end{theorem}

Moreover, Scholl provides a ``comparison theorem'' where the characteristic polynomial of $F$ is shown to be the characteristic polynomial of Frobenius acting on the cohomology of an explicit $\ell$-adic representation. We recall the construction. Let $N\geq 3,$ $M$ as before, $p\nmid 2M,$ and $\ell\neq p$ be a prime. Consider the following diagram
$$
\begin{tikzcd}
{E^{\univ}|_{Y(N)\left[\frac{1}{M\ell}\right]}} \arrow[d, "{\pi|_{Y(N)\left[\frac{1}{M\ell}\right]}}"'] & {Y\left[\frac{1}{\ell}\right]} \arrow[ld, "{h|_{Y\left[\frac{1}{\ell}\right]}}"] \arrow[rr, "j'"] &                                    & {X\left[\frac{1}{\ell}\right]} \arrow[ld, "h"] \arrow[dd, "c"] \\
{Y(N)\left[\frac{1}{M\ell}\right]} \arrow[rr, "j"', hook]                                               &                                                                                                   & {X(N)\left[\frac{1}{M\ell}\right]} &                                                                \\
&                                                                                                   &                                    & {\Spec\Z\left[\frac{1}{M\ell}\right]}                         
\end{tikzcd}
$$
In this notation, if $k\geq 3$ and $\mathcal{F}_{k-2}:=\Sym^{k-2}(R^1(\pi|_{Y(N)\left[\frac{1}{M\ell}\right]})_\ast\Q_\ell),$ then $\mathcal{W}_{k-2}:=R^1c_\ast(h^\ast j_\ast\mathcal{F}_{k-2})$ is a lisse $\ell$-adic sheaf on $\Spec\Z\left[\frac{1}{M\ell}\right].$ The associated $\Gal(\overline{\Q}/\Q)$ representation is unramified at $p$ and by choosing a prime $\mathfrak{p}$ of $\overline{\Q}$ lying over $p,$ we have isomorphisms
$$
G(\overline{\Q})\xrightarrow{\sim} G(\overline{\F}_p)\ \ \ \ \ \ \ \text{ and }\ \ \ \ \ \ \ \mathcal{W}_{k-2}(\overline{\Q})\xrightarrow{\sim} \mathcal{W}_{k-2}(\overline{\F}_p)
$$
which take $\Frob_{\mathfrak{p}}$ to the geometric Frobenius $F_p.$ In this notation, Scholl proves the following result.
\begin{prop}\cite[Proposition 4.4]{Scholl2}\label{prop:charF}
    If $G'$ is a closed subgroup scheme of $G,$ then we have
    $$
    \det(T-F_p|\mathcal{W}_{k-2}(\overline{\F_p})^{G'})=\det(T-F|\mathbb{H}^1(X,\Omega^{\bullet}\otimes L)^{G'})\in\Z[T].
    $$
\end{prop}

\begin{remark}
    The proof of Scholl extends to the case where $\Q$ is replaced by an algebraic number field $K$ whose $p$-adic completion $K_p$ is an unramified extension of $\Q_p.$
\end{remark}

\section{Meromorphic modular forms and de Rham cohomology}\label{sec:de-Rham-meromorphic}

Here we obtain a preliminary version of Theorem~\ref{thm:main-ASD}. To this end, in the notation of Section~\ref{sec:de-Rham-weakly-holomorphic}, let $\mathfrak{u}\subset Y(K)$ be a finite subset of closed points, and consider the chain complex
$$
\Omega_{\mathfrak{u}}^\bullet:=[\begin{tikzcd}
\mathcal{E}_{k-2} \arrow[r, "\nabla_{k-2}"] & {\nabla_{k-2}\left(\mathcal{E}_{k-2}\right)+\mathcal{E}_{k-2}\otimes\Omega_{X}^1(\log\mathfrak{u})}
\end{tikzcd}].
$$
If $R$ is a flat $\Z\left[\frac{1}{M\cdot (k-2)!}\right]$-algebra, our desired space $\frac{S_k(X,\star \mathfrak{u},R)}{\partial^{k-1}M_{2-k}(X,\star \mathfrak{u}, R)}$ will be described in terms of the cohomology group $\mathbb{H}^1(X,\Omega^\bullet_{\mathfrak{u}}\otimes R).$ In what follows, we assume that $u\in\mathfrak{u}$ are defined over $R.$

\begin{remark}
    We note that in recent work, Fonseca and Brown \cite{FonsecaBrown} express this space, when $R$ is a field, in terms of the same cohomology group. We write the details here for completeness.
\end{remark}

We first compute the rank of the cohomology group. 

\begin{lemma}\label{lem:hypercohomology-free}
$\mathbb{H}^1(X,\Omega^\bullet_{\mathfrak{u}}\otimes R)$ is a free $R$-module of rank $(2\dim S_k(X)+(k-1)\cdot\#\mathfrak{u})$ and if $i\neq 1,$ then we have that $\mathbb{H}^i(X,\Omega^\bullet_{\mathfrak{u}}\otimes R)=0.$
\end{lemma}

\begin{proof}
First, we denote by $R_{k-2,\mathfrak{u}}$ the skyscraper sheaf supported on $\mathfrak{u}$ whose stalk at each $u\in\mathfrak{u}$ is the stalk $\mathcal{E}_{k-2,u}$ of $\mathcal{E}_{k-2}$ at $u.$ Denote by $\Res_{\mathfrak{u}}$ the residue map over $\mathfrak{u}$
$$
\Res_{\mathfrak{u}}:\mathcal{E}_{k-2}\otimes\Omega_X^1(\log (Z\cup\mathfrak{u}))\to R_{k-2,\mathfrak{u}}.
$$
It is clear that the following diagram commutes.
$$
\begin{tikzcd}
0 \arrow[r] & \mathcal{E}_{k-2} \arrow[r, "\id"] \arrow[d, "\nabla_{k-2}"]                       & \mathcal{E}_{k-2} \arrow[r] \arrow[d, "\nabla_{k-2}"]                                                                 & 0 \arrow[r] \arrow[d]            & 0 \\
0 \arrow[r] & \nabla_{k-2}(\mathcal{E}_{k-2})+\mathcal{E}_{k-2}\otimes\Omega_X^1 \arrow[r, hook] & \nabla_{k-2}(\mathcal{E}_{k-2})+\mathcal{E}_{k-2}\otimes\Omega_X^1(\log\mathfrak{u}) \arrow[r, "\Res_{\mathfrak{u}}"] & {R_{k-2,\mathfrak{u}}} \arrow[r] & 0,
\end{tikzcd}
$$
where $\id$ is the identity map and the first map in the second row is the natural inclusion. This implies that we have a short exact sequence of chain complexes
$$
\begin{tikzcd}
0 \arrow[r] & \Omega^{\bullet} \arrow[r,hook] & \Omega^{\bullet}_{\mathfrak{u}} \arrow[r] & {R_{k-2,\mathfrak{u}}[-1]} \arrow[r] & 0,
\end{tikzcd}
$$
where $R_{k-2,\mathfrak{u}}[-1]:=[0\longrightarrow R_{k-2,\mathfrak{u}}].$ This induces a long exact sequence on hypercohomology
$$
\begin{tikzcd}[row sep=2ex, column sep=3ex]
             & 0 \arrow[r]                                                        & {\mathbb{H}^0(X,\Omega^\bullet\otimes R)} \arrow[r]                & {\mathbb{H}^0(X,\Omega_{\mathfrak{u}}^\bullet\otimes R)} \arrow[r] & 0                                         \\
& \arrow[r] & {\mathbb{H}^1(X,\Omega^\bullet\otimes R)} \arrow[r]                & {\mathbb{H}^1(X,\Omega_{\mathfrak{u}}^\bullet\otimes R)} \arrow[r] & {H^0(X,R_{k-2,\mathfrak{u}}\otimes R)} \arrow[r]                   & {\mathbb{H}^2(X,\Omega^\bullet\otimes R)} \\
& \arrow[r] & {\mathbb{H}^2(X,\Omega_{\mathfrak{u}}^\bullet\otimes R)} \arrow[r] & {H^1(X,R_{k-2,\mathfrak{u}}\otimes R)} \arrow[r]                   & 0                                                                  &                                          
\end{tikzcd}
$$
By (1) of Theorem~\ref{thm:SchollDecomp2}, we have that $\mathbb{H}^0(X,\Omega^\bullet\otimes R)=\mathbb{H}^2(X,\Omega^\bullet\otimes R)=0.$ Clearly $H^0(X,R_{k-2,\mathfrak{u}}\otimes R)=\bigoplus\limits_{u\in\mathfrak{u}}\mathcal{E}_{k-2,u}\otimes R.$ Moreover, since $R_{k-2,\mathfrak{u}}$ is a skyscraper sheaf, it is acyclic and we have that $H^1(X,R_{k-2,\mathfrak{u}})=0.$ Therefore, we have that $\mathbb{H}^i(X,\Omega^\bullet_{\mathfrak{u}}\otimes R)=0$ for $i\neq 1$ and a short exact sequence
\begin{equation}\label{eq:sesAlgLog}
\begin{tikzcd}
0 \arrow[r] & {\mathbb{H}^1(X,\Omega^\bullet\otimes R)} \arrow[r] & {\mathbb{H}^1(X,\Omega_{\mathfrak{u}}^\bullet\otimes R)} \arrow[r] & {\bigoplus\limits_{u\in\mathfrak{u}}\mathcal{E}_{k-2,u}\otimes R} \arrow[r] & 0.
\end{tikzcd}
\end{equation}
Since ${\bigoplus\limits_{u\in\mathfrak{u}}\mathcal{E}_{k-2,u}\otimes R}$ is a free $R$-module, the short exact sequence splits. Similarly, from the short exact sequence in (1) of Theorem~\ref{thm:SchollDecomp2} and the fact that $S_k(X,R)$ is free, we have that $\mathbb{H}^1(X,\Omega^\bullet\otimes R)$ is free of rank $2\dim S_k(X,R)$, and therefore, $\mathbb{H}^1(X,\Omega_{\mathfrak{u}}^\bullet\otimes R)$ is a free $R$-module. Moreover, the rank of $\bigoplus\limits_{u\in\mathfrak{u}}\mathcal{E}_{k-2,u}\otimes R$ is clearly $(k-1)\cdot\#\mathfrak{u}$ and this gives the desired statement on the rank.
\end{proof}

If $R$ is a field of characteristic $0,$ we describe $\mathbb{H}^1(X,\Omega^\bullet_{\mathfrak{u}})$ in terms of meromorphic modular forms.

\begin{lemma}\label{lem:hypercohomology-mdular-comparison}
Assume $R$ is a field of characteristic $0$ and let $\mathfrak{u}\subset Y(R)$ be a non-empty finite subset of closed points. Denote by $U$ the open affine subset $X\setminus\mathfrak{u}$ and let $j:U\hookrightarrow X$ be the inclusion. The inclusion map $\Omega_{\mathfrak{u}}^\bullet\hookrightarrow j_\ast(\Omega^\bullet|_U)$ induces an isomorphism
$$
\mathbb{H}^1(X,\Omega^\bullet_{\mathfrak{u}}\otimes R)\cong\frac{S_k(X,\star \mathfrak{u},R)}{\partial^{k-1}M_{2-k}(X,\star \mathfrak{u}, R)}.
$$
\end{lemma}

\begin{proof}
    If $R$ is a field of characteristic $0,$ we have that (see Corollary~\ref{cor:jfp1} which is a modification of Corollary~2.2.6 of \cite{kedlaya})
    $$
    \mathbb{H}^1(X,\Omega^\bullet_{\mathfrak{u}}\otimes R)\cong \mathbb{H}^1(X,j_\ast(\Omega_{\mathfrak{u}}^\bullet|_U)\otimes R)=\mathbb{H}^1(U,\Omega_{\mathfrak{u}}^\bullet|_U\otimes R).
    $$
    The spectral sequence for hypercohomology filtered by the $(k-2)$-th symmetric power of the Hodge filtration gives (see proof of part (i) of Theorem 2.7 in \cite{Scholl2}) an exact sequence
    $$
    \begin{tikzcd}
             & 0 \arrow[r]                                           & {H^0(U,\boldsymbol{\omega}^{2-k}\otimes R)} \arrow[r] & {H^0(U,\boldsymbol{\omega}^{k-2}\otimes\Omega^1_U\otimes R)} \arrow[r] & {\mathbb{H}^1(U,\Omega_{\mathfrak{u}}^\bullet|_U\otimes R)} \\
& \arrow[r] & {H^1(U,\boldsymbol{\omega}^{2-k}\otimes R)} \arrow[r] & {H^1(U,\boldsymbol{\omega}^{k-2}\otimes\Omega^1_U\otimes R)} \arrow[r] & 0                                                                  &                                                            
    \end{tikzcd}
    $$
    Since $U$ is affine and both $\boldsymbol{\omega}^{2-k}$ and $\boldsymbol{\omega}^{k-2}\otimes\Omega^1_U$ are coherent (in fact, locally free) sheaves of $\O_U$-modules, they are acyclic (see Section 44, Corollary 1 of \cite{serreFaisceaux}) and $H^1(U,\boldsymbol{\omega}^{k-2}\otimes\Omega^1_U\otimes R)=H^1(U,\boldsymbol{\omega}^{2-k}\otimes R)=0.$ Moreover, the proof of (i) of Theorem 2.7 gives that the map $H^0(U,\boldsymbol{\omega}^{2-k}\otimes R)\to H^0(U,\boldsymbol{\omega}^{k-2}\otimes\Omega_U^1\otimes R)$ is the map $\partial^{k-1},$ and we have our result.
\end{proof}

\begin{remark}
If $R=\mathbb{C},$ a theorem of Deligne (see Corollary 3.15 of \cite{DelDE}) shows that 
    $$
    \mathbb{H}^1(X,\Omega^\bullet_{\mathfrak{u}}\otimes \C)\cong \mathbb{H}^1(X,j_\ast(\Omega_{\mathfrak{u}}^\bullet|_U)\otimes \C)=\mathbb{H}^1(U,\Omega_{\mathfrak{u}}^\bullet|_U\otimes \C).
    $$
\end{remark}

In order to work in the $p$-adic setting, note that if $L/\Q_p$ is an unramified extension of degree $r$ with ring of integers $\mathfrak{o}_L$ and Frobenius morphism $\sigma,$ the $F$-crystal structure on $(\mathcal{E}_{k-2}^\infty,\nabla_{k-2}^\infty)$ induces a $\sigma$-linear Frobenius morphism $F$ on $\mathbb{H}^1(X,\Omega^\bullet_{\mathfrak{u}}\otimes\mathfrak{o}_L).$ Moreover, we assume that the reductions of $u\in\mathfrak{u}$ in $\F_{p^r}$ are distinct and $p\nmid M\cdot (k-2)!$. We determine the characteristic polynomial of the linear map $F^r.$

\begin{lemma}\label{lem:charFLog}
    If $G'$ is a closed subgroup scheme of $G$ that preserves $\mathfrak{u},$  then we have that\footnote{By $R_{k-2,\mathfrak{u}}$ and $F|R_{k-2,\mathfrak{u}}$ we mean the direct sum of stalks $\Sym^{k-2}\mathcal{E}_{u}^{\infty}$ as in Subsection~\ref{subsec:pAdicTheory} and the corresponding Frobenius on this $\mathfrak{o}_L$-module.}
    $$
    \det(T-F^r|\mathbb{H}^1(X,\Omega^\bullet_{\mathfrak{u}}\otimes\mathfrak{o}_L)^{G'})=\det(T-F^r|\mathbb{H}^1(X,\Omega^\bullet\otimes\mathfrak{o}_L)^{G'})\cdot\det(T-p^r\cdot F^r|R_{k-2,\mathfrak{u}}^{G'}\otimes\mathfrak{o}_L).
    $$
\end{lemma}

\begin{proof}
    By the proof of Lemma~\ref{lem:hypercohomology-free}, we have a short exact sequence 
    \begin{equation}\label{eq:sesAlgLogFrob}
    \begin{tikzcd}
    0 \arrow[r] & {\mathbb{H}^1(X,\Omega^\bullet\otimes \mathfrak{o}_L)} \arrow[r] & {\mathbb{H}^1(X,\Omega_{\mathfrak{u}}^\bullet\otimes \mathfrak{o}_L)} \arrow[r] & {\bigoplus\limits_{u\in\mathfrak{u}}\mathcal{E}_{k-2,u}\otimes \mathfrak{o}_L(-1)} \arrow[r] & 0,
    \end{tikzcd}
    \end{equation}
    where $F$ acts on $\bigoplus\limits_{u\in\mathfrak{u}}\mathcal{E}_{k-2,u}\otimes \mathfrak{o}_L(-1)$ as $pF$ on $\bigoplus\limits_{u\in\mathfrak{u}}\mathcal{E}_{k-2,u}\otimes \mathfrak{o}_L.$ This follows from the construction of the action of $F$ on the cohomology and from the fact that if $V$ is a neighborhood of $u,$ $\phi$ is an admissible lift of the absolute Frobenius map on $V\otimes\F_{p^r}$ to $V$ and $t$ is a local uniformizer at $u,$ then $\phi^\ast(\frac{dt}{t})=p\frac{dt}{t}$ (see 5.6 of \cite{Scholl1}).  This short exact sequence is compatible with the action of $G'$ since $G'$ stabilizes $\mathfrak{u}$ (see 7.2 of \cite{Scholl1}) and our claim follows.
\end{proof}

Moreover, in the notation of Proposition~\ref{prop:charF}, we describe $\det(T-F^r|\H^1(X,\Omega^\bullet_{\mathfrak{u}})^{G'})$ as the trace of Frobenius acting on the cohomology of an $\ell$-adic representation.

\begin{lemma}\label{lem:ladicChar}
    If $G'$ is a closed subgroup scheme of $G$ that preserves $\mathfrak{u},$ then we have
    $$
    \det(T-F^r|\H^1(X,\Omega^\bullet_{\mathfrak{u}}\otimes\mathfrak{o}_L)^{G'})=\det(T-F_p^r|\mathcal{W}_{k-2}(\overline{\F_p})^{G'})\cdot\det(T-F_p^r|\iota_\ast\iota^\ast\mathcal{G}_{k-2}(\overline{\F_p})(-1)),
    $$
    where $\mathcal{G}_{k-2}$ is the sheaf $h^\ast j_\ast \mathcal{F}_{k-2}$ on $X\left[\frac{1}{\ell}\right]$ in the notation of Subsection~\ref{subsec:alg-theory} and where $\iota:\mathfrak{u}\to X$ is the natural embedding.
\end{lemma}

\begin{proof}
    This follows immediately from Proposition~\ref{prop:charF} and Lemma~\ref{lem:charFLog}.
\end{proof}

We are now ready to prove a weak version of Theorem~\ref{thm:main-ASD}. 

\begin{prop}\label{prop:main-ASD-weak}
    Let $K$ be an algebraic number field, $k\geq 3$ and $\mathfrak{u}\subset Y(K).$ If  $d:=2\dim S_k(X)+(k-1)\cdot\#\mathfrak{u}$ and $p$ is a good prime for $(M,k,K,\mathfrak{u})$ with $[K_p:\Q_p]=r,$ then there exists a polynomial 
    $$
    P_{k,p^r,\mathfrak{u}}(T)=\sum\limits_{i=0}^{d} A_{i,p^r,\mathfrak{u}} T^i \in \Z[T]
    $$ 
    such that for any $f=\sum\limits_{n\geq 1} a_f(n)t^n\in S_k(X_\Gamma,\star\mathfrak{u}, \mathfrak{o}_K\left[\frac{1}{M},u,\frac{1}{u}\right])$ and any $m\geq 1, s\geq 0,$ we have
    \begin{equation}
    \sum\limits_{i=0}^d p^{r(k-1)i}\cdot A_{i,p^r,\mathfrak{u}}\cdot\gamma_p^{mp^s(p^{ri}-1)/(p-1)}\cdot a_{f}\left(\frac{mp^s}{p^{ri}}\right) \equiv 0\pmod{p^{(k-1)s+j_{f,p}}}  
    \end{equation}
    for some $j_{f,p}\in\Z.$
\end{prop}

\begin{proof}
Let $G'=1,$ and denote by $F\otimes\id$ the extension of $F$ to $\mathbb{H}^1(X,\Omega^\bullet_{\mathfrak{u}})\otimes K_p,$ where we fix an embedding $K\hookrightarrow K_p.$ By Lemma~\ref{lem:hypercohomology-mdular-comparison} and the compatibility of the constructions with base change, we have that
$$
\mathbb{H}^1(X,\Omega^\bullet_{\mathfrak{u}})\otimes K_p\cong\frac{S_k(X,\star \mathfrak{u},K_p)}{\partial^{k-1}M_{2-k}(X,\star \mathfrak{u}, K_p)}.
$$
The characteristic polynomial of the linear map $F^r\otimes\id$ on $\mathbb{H}^1(X,\Omega_{\mathfrak{u}}^\bullet)\otimes K_p$ lands in $\Z[T]$ by Lemma~\ref{lem:charFLog} and Proposition~\ref{prop:charF}. Moreover, we have that $F$ acts on the $t$-expansion by
$$
F\left(\sum\limits_{n\geq 1}a_f(n)t^n\right)\equiv p^{k-1}\sum\limits_{n\geq 1}\sigma(a_f(n))\gamma_p^{n}t^{pn}\mod \left(q\frac{d}{dq}\right)^{k-1}(\mathfrak{o}_{K_p}[[t]]\otimes K_p),
$$
where $\sigma$ is the Frobenius endomorphism of $K_p.$ If we denote the characteristic polynomial of $F^r$ by $P_{k,p^r\mathfrak{u}},$ then we have that
$$
P_{k,p^r,\mathfrak{u}}(F)\left(\sum\limits_{n\geq 1}a_f(n)t^n\right)\equiv 0\mod \left(q\frac{d}{dq}\right)^{k-1}(\mathfrak{o}_{K_p}[[t]]\otimes {K_p}),
$$
and this is precisely our statement.
\end{proof}
\begin{remark}\label{rem:BaseChangeMeaning}
    $\mathbb{H}^1(X,\Omega^\bullet_{\mathfrak{u}}\otimes K)\cong \mathbb{H}^1(X,\Omega^\bullet_{\mathfrak{u}}\otimes\mathfrak{o}_K\left[\frac{1}{M},u,\frac{1}{u}\right])\otimes K$ implies that for a meromorphic modular form $f$ with Fourier coefficients in $K,$ there exists $c\in \mathfrak{o}_K$ such that $cf$ has Fourier coefficients in $\mathfrak{o}_K\left[\frac{1}{M},u,\frac{1}{u}\right].$ 
\end{remark}

\begin{remark}\label{rem:computing-P-polynomial}
    {By Lemma~\ref{lem:ladicChar}, we can write the polynomial $P_{k,p^r,\mathfrak{u}}(T)$ as a product of polynomials 
    $$
    P_{k,p^r,\mathfrak{u}}(T)=\det(T-F_p^r|\mathcal{W}_{k-2}(\overline{\F_p})^{G'})\prod\limits_{u\in\mathfrak{u}}\det(T-F_p^r|\left((\iota_u)\ast(\iota_u)^\ast\mathcal{G}_{k-2}(\overline{\F_p})(-1)\right)^{G'}),
    $$
    where $\mathcal{G}_{k-2}$ is as in Lemma~\ref{lem:ladicChar} and $\iota_u$ is the embedding of $u$ into $X.$
    The first polynomial $\det(T-F_p^r|\mathcal{W}_{k-2}(\overline{\F_p})^{G'})$ is the same as that of \cite[\S4]{Scholl2} and has roots of modulus $p^{\frac{r(k-1)}{2}}$ under any embedding $\overline{\Q_\ell}\hookrightarrow\C.$  On the other hand, for $u\in\mathfrak{u}$ and $j(E^{\univ}_{\kappa(u)}/\F_p)\not\in\{0,1728\},$ we have that 
    $$
    \det(T-F_p^r|\left((\iota_u)\ast(\iota_u)^\ast\mathcal{G}_{k-2}(\overline{\F_p})(-1)\right)^{G'})=\prod_{j=0}^{k-2}(T- p^r\cdot \alpha^j (p^r/\alpha)^{k-2-j}),
    $$
    where 
    $$
    \alpha+\frac{p^r}{\alpha}=p^r+1-\#E^{\univ}_{\kappa(u)}(\F_{p^r}).
    $$
    Therefore, in contrast to $\det(T-F_p|\mathcal{W}_{k-2}(\overline{\F_p})^{G'})$, the roots have modulus $p^{\frac{rk}{2}}.$ When $j(E^{\univ}_{\kappa(u)}/\F_p)\in\{0,1728\},$ see Subsection~\ref{subsec:Bonsich} for an example.
    }
\end{remark}

\section{An integral computation}\label{sec:integral-shift-computation}

In order to derive Theorem~\ref{thm:main-ASD} from Proposition~\ref{prop:main-ASD-weak}, it remains to determine the shift $j_{f,p}.$ This shift arises from the identification in Lemma~\ref{lem:hypercohomology-mdular-comparison}, where one passes to the field $K$ (in addition, see Remark~\ref{rem:BaseChangeMeaning}). Here we provide an integral computation that allows us to track the ``division by $p$''. In what follows, we follow the computations in Section~2 of \cite{kedlaya}.

For $m\geq 0,$ denote by $\Omega_{\mathfrak{u}}^\bullet(m\mathfrak{u})$ the chain complexes obtained by tensoring each term with $\O_X(m\mathfrak{u}).$ In this notation, we have the following proposition.

\begin{prop}\label{prop:quasi}
    If $m\geq 0,$ then the cokernels of the maps on the $0$th and $1$st homology sheaves induced by the natural maps of complexes
    $$\Omega^\bullet_{\mathfrak{u}}\to\Omega_{\mathfrak{u}}^\bullet(m\mathfrak{u})
    $$
    are killed by $1$ and $m!$ respectively.
\end{prop}

\begin{proof}

This can be verified stalkwise. Let $R=\mathfrak{o}_K\left[\frac{1}{M}\right].$ Without loss of generality (see Remark 2.2.3 and the proof of Theorem 2.2.5 of \cite{kedlaya}), we can replace $X$ by $\Spec R[T]$ with $u\in\mathfrak{u}$ replaced with the point $T=0.$ Let $\omega$ and $\xi$ be generators of $\mathcal{E}_u,$ and locally, write the Gauss-Manin connection as
$$
\nabla=d+\gamma\otimes dT,
$$
where $\gamma\in M_{2\times 2}(R[T]).$

For the $0$th homology sheaf, the claim is the following. If $s$ is a local section of $\mathcal{E}_{k-2}\otimes_{R[T]}\frac{1}{T^m}R[T]$ such that $\nabla_{k-2}(s)=0,$ then $s$ is a local section of $\mathcal{E}_{k-2}.$ This is obvious.

For the $1$st homology sheaf, the claim in the lemma is the following: Let $s_1$ and $s_2$ be local sections of $\mathcal{E}_{k-2}\otimes_{R[T]}\frac{1}{T^m}R[T]$ and $P_1(T),P_2(T)\in R[T].$ If we have a local section
$$
g=\frac{P_1(T)}{T^m}\cdot\nabla_{k-2}(s_1)+s_2\cdot\frac{P_2(T)}{T^{m+1}}dT,
$$
then there exists a local section $t_1$ of $\mathcal{E}_{k-2}\otimes_{R[T]}\frac{1}{T^m}R[T],$ a local section $t_2$ of $\mathcal{E}_{k-2},$ and $P_3(T)\in R[T]$ such that
$$
(m!)\cdot g=\nabla_{k-2}(t_1)+t_2\cdot\frac{P_3(T)dT}{T}.
$$
We prove this by induction. If $m=0,$ there is nothing to prove. Suppose the statement is true for $m-1\geq 0,$ and consider $g$ as above. We then have
$$
g=\nabla_{k-2}(s_1)\cdot\frac{P_1(T)}{T^m}+\sum\limits_{i=0}^{k-2}\omega^{k-2-i}\xi^i\cdot Q_{i}(T)\frac{dT}{T^{m+1}},
$$
where $Q_i(T)\in R[T].$ For $0\leq i\leq k-2,$ write $Q_i=A_i+TB_i(T), A_i\in R, B_i\in R[T].$ By the Leibniz rule for $\nabla_{k-2},$ we have that
$$
\nabla_{k-2}\left(-\sum\limits_{i=0}^{k-2}\omega^{k-2-i}\xi^i T^{-m}\right)=m\sum\limits_{i=0}^{k-2}\omega^{k-2-i}\xi^i\frac{dT}{T^{m+1}}+t\otimes\frac{dT}{T^m},
$$
where $t$ is a local section of $\mathcal{E}_{k-2}\otimes\frac{1}{T^m}dT.$ This implies that
$$
mg=\nabla_{k-2}\left(-\sum\limits_{i=0}^{k-2}A_i\omega^{k-2-i}\xi^i T^{-m}\right)+\nabla_{k-2}(ms_1)\cdot\frac{P_1(T)}{T^m}+s_3\cdot\frac{P_3(T)}{T^m}dT,
$$
where $s_3$ is a local section of $\mathcal{E}_{k-2}$ and $P_3(T)\in R[T].$ Now, the induction hypothesis applies to $h=\nabla_{k-2}(ms_1)\cdot\frac{P_1(T)}{T^m}+s_3\cdot\frac{P_3(T)}{T^m}dT,$ that is, $(m-1)!h$ can be written in the required form. This implies that $m!g$ can be written in the required form.
\end{proof}

As a corollary, we obtain some control on the shift $j_{f,p}$ in Proposition~\ref{prop:main-ASD-weak}.

\begin{cor}\label{cor:jfp1}
    If $K$ is a field of characteristic $0,$ then we have that $\mathbb{H}^1(X,\Omega^\bullet_{\mathfrak{u}})\otimes K\cong \mathbb{H}^1(U,\Omega^\bullet|_U)\otimes K.$ Moreover, if $f\in S_k(X,\star\mathfrak{u},\mathfrak{o}_K\left[\frac{1}{M},u\right]), $ then the $j_{f,p}$ in Proposition~\ref{prop:main-ASD-weak} satisfies the inequality 
    $$
    j_{f,p}\geq -\ord_p((-\ord_{\mathfrak{u}}(f)-1)!).
    $$
\end{cor}

\begin{proof}
If $K$ is a field, then Proposition~\ref{prop:quasi} tells us that for all $m\geq 0,$  we have that $\Omega^\bullet_{\mathfrak{u}}\otimes K$ and $\Omega^\bullet_{\mathfrak{u}}(m\mathfrak{u})\otimes K$ are quasi-isomorphic as chain complexes. The statement then follows by noting that $\mathbb{H}^1(U,\Omega^\bullet|_U)\otimes K=\mathbb{H}^1(X,\varinjlim\limits_{m} \Omega^\bullet_{\mathfrak{u}}(m\mathfrak{u})\otimes K)$ and that the construction of hypercohomology over Noetherian spaces commutes with direct limits (see III.2.9 of \cite{HartshorneAG}).

To compute $j_{f,p},$ if $m\geq 0$ and $q\in\{0,1\},$ then let $\mathcal{H}^q_m$ denote the homology sheaves of $\Omega^\bullet_{\mathfrak{u}}(m\mathfrak{u})$.  We have the Leray spectral sequence (for example, see Theorem 14.14 of \cite{BottTu})
$$
E_{2,m}^{p,q}=H^p(X,\mathcal{H}^q_m)\Rightarrow\mathbb{H}^{p+q}(X,\Omega^\bullet_{\mathfrak{u}}(m\mathfrak{u})).
$$
It is clear that $E_{2,m}^{p,q}=0$ if $q\not\in\{0,1\}$ and that $E_{\infty,m}=E_{3,m}.$ Proposition~\ref{prop:quasi} tells us that the natural maps $E_{\infty,0}^{1,0}\to E_{\infty,m}^{1,0}$ and $E_{\infty,0}^{0,1}\to E_{\infty,m}^{1,0}$ are killed by $1$ and $m!$ respectively. This implies that $m!$ kills the cokernel of the natural map $\mathbb{H}^1(X,\Omega^\bullet_{\mathfrak{u}})\to\mathbb{H}^1(X,\Omega_{\mathfrak{u}}^{\bullet}(m\mathfrak{u})).$ Therefore, if $m:=-\ord_{\mathfrak{u}}(f)-1,$ then $m!f$ lies in the image of $\mathbb{H}^1(X,\Omega^\bullet_{\mathfrak{u}})$ in $\mathbb{H}^1(U,\Omega^\bullet|_U)$ which gives our claim.

\end{proof}

\begin{remark}
    If the poles are located at cusps, the method of proof above gives the condition of weakly exact cusp form in Theorem~\ref{thm:KSChar}. The major difference is that $\nabla$ has logarithmic singularities at the cusps which alters the computation in Proposition~\ref{prop:quasi}.
\end{remark}

When we have $\dim S_k(X)=0,$ we require a stronger bound on $j_{f,p}$ in order to obtain the extra $p^{k-i}$ in Theorem~\ref{thm:main-ASD}. To this end, following Scholl, we consider a modified version of the chain complex $\Omega^\bullet_{\mathfrak{u}}.$ Namely, in the notation of Subsection~\ref{subsec:alg-theory}, define
$$
\Omega_{\mathfrak{u},p}^\bullet:=[\begin{tikzcd}
p\mathcal{E}_{k-2,p} \arrow[r, "\nabla_{k-2}"] & {\nabla_{k-2}\left(p\mathcal{E}_{k-2,p}\right)+\mathcal{E}_{k-2,p}\otimes\Omega_{X}^1(\log\mathfrak{u})}
\end{tikzcd}].
$$
We clearly have that the following diagram commutes
$$
\begin{tikzcd}[row sep=2ex, column sep=3ex]
0 \arrow[r] & p\mathcal{E}_{k-2,p} \arrow[r, "\id"] \arrow[d, "\nabla_{k-2}"]                       & p\mathcal{E}_{k-2,p} \arrow[r] \arrow[d, "\nabla_{k-2}"]                                                                 & 0 \arrow[r] \arrow[d]            & 0 \\
0 \arrow[r] & \nabla_{k-2}(p\mathcal{E}_{k-2,p})+\mathcal{E}_{k-2,p}\otimes\Omega_X^1 \arrow[r, hook] & \nabla_{k-2}(p\mathcal{E}_{k-2,p})+\mathcal{E}_{k-2,p}\otimes\Omega_X^1(\log\mathfrak{u}) \arrow[r, "\Res_{\mathfrak{u}}"] & {R_{k-2,\mathfrak{u},p}} \arrow[r] & 0,
\end{tikzcd}
$$
where $R_{k-2,\mathfrak{u},p}$ is the skyscraper sheaf supported on $\mathfrak{u}$ with the fiber at $u$ being the fiber of $\mathcal{E}_{k-2,p}.$
The proofs of Lemmas~\ref{lem:hypercohomology-free} and ~\ref{lem:hypercohomology-mdular-comparison}, Proposition~\ref{prop:quasi}, and Corollary~\ref{cor:jfp1} go through almost verbatim to give the following proposition.

\begin{prop}\label{prop:extraP}
    Let $K$ be an algebraic number field, $k\geq 3, \mathfrak{u}\subset Y(K),$ and $p$ a good prime for $(M,k,K,\mathfrak{u}).$  If $\dim S_k(X)=0,$ then the following are true.
    \begin{enumerate}
    \item $\mathbb{H}^1(X,\Omega^\bullet_{\mathfrak{u},p})$ is a free $\mathfrak{o}_K\left[\frac{1}{M}\right]$-module of rank $(k-1)\cdot\#\mathfrak{u}$ and if $i\neq 1,$ then $\mathbb{H}^i(X,\Omega^\bullet_{\mathfrak{u},p})=0.$
    \item If $f\in S_k(X,\star\mathfrak{u},\mathfrak{o}_K\left[\frac{1}{M},\mathfrak{u}\right]),$ then 
    $$
    \left(p^{-\ord_{\mathfrak{u}}(f)-1}f \mod \left(p\cdot q\frac{d}{dq}\right)^{k-1}M_{2-k}\left(X,\star\mathfrak{u},\mathfrak{o}_K\left[\frac{1}{M},\mathfrak{u}\right]\right)\right)\in \H^1(X,\Omega^\bullet_{\mathfrak{u},p}).
    $$
    \item If $f\in S_k(X,\star\mathfrak{u},\mathfrak{o}_K\left[\frac{1}{M},\mathfrak{u}\right]),$ then the $j_{f,p}$ in Proposition~\ref{prop:main-ASD-weak} satisfies the inequality
    $$
    j_{f,p}\geq k+\ord_{\mathfrak{u}}(f).
    $$
    \end{enumerate}
\end{prop}

Finally, in order to prove Corollary~\ref{cor:ASD-basis}, we require the effect of the quasi-isomorphisms in Proposition~\ref{prop:quasi} on the Hodge filtrations. To this end, we first recall some facts regarding this filtration (see proof of Theorem 2.7 of \cite{Scholl2} or Section 3 of \cite{KS16}). The Hodge filtration on $\mathcal{E}$ is defined as $F_{\Hdg}^1\mathcal{E}=\boldsymbol{\omega}$ and $F_{\Hdg}^0\mathcal{E}=\mathcal{E}.$ On $\mathcal{E}_{k-2},$ we denote by $F^\bullet$ the natural filtration induced on $\mathcal{E}_{k-2}$ by $F_{\Hdg}.$ In this notation, we have Griffiths transversality
\begin{equation}\label{eq:Griffiths}
\nabla_{k-2}(F^i\mathcal{E}_{k-2})\subset F^{i-1}\mathcal{E}_{k-2}\otimes\Omega_X^1(\log Z).
\end{equation}
Therefore, we have the natural definition of $F^\bullet$ on $\Omega_{\mathfrak{u}}^\bullet$ defined by
\begin{align*}
&F^i\Omega_{\mathfrak{u}}^0 = F^i\mathcal{E}_{k-2} \\
&F^i\Omega_{\mathfrak{u}}^1 = \Omega_{\mathfrak{u}}^1\cap(F^{i-1}\mathcal{E}_{k-2}\otimes\Omega_X^1(\log(Z\cup\mathfrak{u})))
\end{align*}
with the obvious analogous definition for $\Omega^\bullet_{\mathfrak{u},p}.$
Moreover, the Hodge filtration on $\mathcal{E}$ is clearly compatible with the Hodge filtration on $\mathcal{E}_x\cong H^1_{dR}(E^{\univ}_x)$ for any $x\in X.$ Therefore, from (\ref{eq:Griffiths}) and the proof of Proposition~\ref{prop:quasi}, we have that
\begin{equation}\label{eq:ResShift}
\Res(F^i\Omega_{\mathfrak{u}}^1)\subset F^{i-1}H^1_{dR}(E^{\univ}_u)
\end{equation}
with a similar statement for $\Omega^\bullet_{\mathfrak{u},p}.$

In this notation, we have the following relation between pole orders of meromorphic modular forms and the Hodge filtration. 

\begin{lemma}\label{lem:quasiHodge}
Let $k\geq 3,$ $K$ an algebraic number field and $u\in Y(K).$ If $\dim S_k(X)=0$ and $f\in S_k(X,\star u,K),$ then we have $\phi(f)\in F^{k-1+\ord_u(f)}(R_{k-2,u}\otimes K),$ where $\phi$ is
the isomorphism
$$
\phi:\frac{S_k(X,\star u,K)}{\partial^{k-1}M_{2-k}(X,\star u,K)}\xrightarrow{\sim}\H^1(U,\Omega^\bullet|_U\otimes K)\xrightarrow{\sim} \H^1(X,\Omega^\bullet_{u}\otimes K)\xrightarrow[Res]{\sim} R_{k-2,u}\otimes K.
$$ 
\end{lemma}

\begin{proof}
    If $f\in S_k(X,(k-2)\cdot u,K)$\footnote{If $\delta\geq 0,$ $S_k(X,\delta\cdot u,K)$ is the subspace of forms in $S_k(X,\star u,K)$ with pole orders at $u$ at most $\delta.$}, then the image of $f$ under 
    $$
    \frac{S_k(X,\star u,K)}{\partial^{k-1}M_{2-k}(X,\star u,K)}\xrightarrow{\sim}\H^1(U,\Omega^\bullet|_U\otimes K)
    $$
    lies in $\frac{H^0(U,\boldsymbol{\omega}^{k-2}\otimes\Omega_X^1)}{\partial^{k-1}H^0(U,\boldsymbol{\omega}^{2-k})}$ (note that $U$ is affine). By the proof of Proposition~\ref{prop:quasi} and (\ref{eq:Griffiths}), the image of $f$ in $\H^1(X,\Omega^\bullet_{\mathfrak{u}}\otimes K)$ under the quasi-isomorphism is locally represented by a section of $F^{k+\ord_u(f)}\Omega^1_u\otimes K.$ The claim follows from (\ref{eq:ResShift}).
    
\end{proof}

\section{CM and de Rham cohomology}\label{sec:CM-de-Rham}

To obtain Corollary~\ref{cor:ASD-basis} from Theorem~\ref{thm:main-ASD}, we show that when $\dim S_k(X)=0,$ the CM structure on the fibers at poles endows the meromorphic modular forms with additional structure.

To this end, we first show that we can work at each pole $u\in\mathfrak{u}$ separately. 

\begin{lemma}\label{lem:separation}
    Let $K$ be an algebraic number field, $k\geq 3,$ and $\mathfrak{u}\subset Y(K).$ If $\dim S_k(X)=0,$ then there exists a basis 
    $$
    \{f_{u,i} | u \in \mathfrak{u}, 0\leq i\leq k-2\}
    $$
    of
    $\frac{S_k(X,\star\mathfrak{u},K)}{\partial^{k-1}M_{2-k}(X,\star\mathfrak{u},K)}$ where each $f_{u,i}\in S_k(X,(i+1)\cdot u,K).$\footnote{In this lemma, the $f_{u,i}$ are chosen so that the pole has order exactly $i+1.$}
\end{lemma}

\begin{proof}
    We first prove that for a pole of order $n\geq 1$ and $u\in\mathfrak{u},$ there exists $f\in S_k(X,nu,K).$ To this end, consider $\mathcal{L}:=\omega^{k-2}\otimes\O_X(nu).$ Applying the Riemann--Roch theorem (for example, see IV.1 of \cite{HartshorneAG}) to $\mathcal{L}^{-1}$  gives that
    $$
    \dim H^0(X,\mathcal{L}\otimes\Omega_X^1) - \dim H^0(X,\mathcal{L}^{-1}) = \deg(\mathcal{L}) - 1 + g,
    $$
    where $g$ is the genus of $X.$ Since $\mathcal{L}$ is ample (see VII.3.4 of \cite{DR}), we have that $\dim H^0(X,\mathcal{L}^{-1})=0$ and thus
    $$
    \dim H^0(X,\omega^{k-2}\otimes\Omega_X^1(nu)) = (k-2)\deg\omega+(n-1)+g\geq n.\footnote{In fact, by our hypothesis that $\dim S_k(X)=0,$ one can compute the dimension directly.}
    $$
    Therefore, for each $u\in\mathfrak{u}$ and $0\leq i\leq k-2,$ we can choose $f_{u,i}$ with exact pole order $i+1$ at $u$ and that vanishes at the cusps. The $f_{u,i}$ are clearly linearly independent and by Lemma~\ref{lem:hypercohomology-free}, they form a basis.
\end{proof}

By Lemma~\ref{lem:separation}, it suffices to work with the case where $\mathfrak{u}=\{u\}.$ In order to obtain a basis which does not depend on $p,$ we make use of the commutativity of $F$ and CM structure.

\begin{lemma}\label{lem:CMMod}
    Let $K$ be an algebraic number field, $k\geq 3$ such that $\dim S_k(X)=0,$ and $u\in Y(K)$ such that $E^{\univ}_{\kappa(u)}$ is a CM elliptic curve by an order $\O$ in an imaginary quadratic subfield of $K.$ If $p$ is good for $(M,k,K,u)$ and $E^{\univ}_{\kappa(u)}$ has good reduction, then there exists an action of $\O$ on $\frac{S_k(X,\star u,K)}{\partial^{k-1}M_{2-k}(X,\star u,K)}$ that commutes with $F.$
\end{lemma}

\begin{proof}
    The short exact sequence (\ref{eq:sesAlgLogFrob}) and part (1) of Theorem~\ref{thm:SchollDecomp2} shows that the residue map
    $$
    \begin{tikzcd}
    {\H^1(X,\Omega^\bullet_{u}\otimes K)} \arrow[r, "\Res"] & {R_{k-2,u}(-1)\otimes K}
    \end{tikzcd}
    $$
    is an isomorphism. Fixing an embedding $K\hookrightarrow K_p,$ since $\O$ acts on $R_{k-2,u},$ we have an action of $\O$ on $\H^1(X,\Omega^\bullet_{u}\otimes K_p)$ (note that these constructions are compatible with flat base-change). Moreover, since the actions of $\O$ and $F$ commute on $R_{k-2,u},$ they commute on $\H^1(X,\Omega^\bullet_{u}\otimes K_p)$ and the claim follows by Lemma~\ref{lem:hypercohomology-mdular-comparison}.
\end{proof}

Using this commutativity, we obtain a ``global'' decomposition of the basis. 

\begin{lemma}\label{lem:indDecomp}
    Let $K$ be an algebraic number field, $k\geq 3$ with $\dim S_k(X)=0,$ and $u\in Y(K)$ such that $E_{\kappa(u)}^{\univ}$ is a CM elliptic curve by an order $\O$ in an imaginary quadratic subfield of $K.$  The basis $\{f_{u,i} | 0\leq i\leq k-2\}$ in Lemma~\ref{lem:separation} can be chosen such that for all primes $p$ that are good for $(M,k,K,u)$ and for which $E^{\univ}_{\kappa(u)}$ has good reduction, we have \footnote{By abuse of notation, we are denoting the classes of $f_{u,i}$ by $f_{u,i}$ as well.}
    $$
    F(f_{u,i}) =  p\cdot c_{u,i}\cdot f_{u,\tilde{\sigma}(i)}, 
    $$
    where  $c_{u,i}\in p^{k-2-i}\mathfrak{o}_{K_p}$  and  $\tilde{\sigma}$ is as in Corollary~\ref{cor:ASD-basis}. 
\end{lemma}

\begin{proof}
    Choose $\alpha\in\O\setminus\Z$ such that $\frac{\sigma(\alpha)}{\alpha}$ is not a root of unity. Then there are $(k-1)$ distinct eigenvalues of the action of $\alpha$ on $R_{k-2,u}$ and they all lie in $K.$ Therefore, there exists a unique (up to scalar multiple) diagonalization of $\alpha$ on $R_{k-2,u}$ over $K.$ Choose $\{f_{u,i}\}$ so that it is such a diagonal basis with $[\alpha](f_{u,i})=\alpha^{k-2-i}\sigma(\alpha)^if_{u,i}$\footnote{Note that the $f_i$ under $\Res$ map to $\omega_1^{k-2-i}\omega_2^i$ where $\omega_1\in F^1_{\Hdg}H^1_{dR}(E)$ and $\omega_1,\omega_2$ are eigenvectors of CM acting on the de Rham cohomology. This explains the form of the eigenvalues. The eigenvalues are defined up to the embedding of $\O$ into $K_p,$ but the argument is independent of this.}. By Lemma~\ref{lem:CMMod}, we have
    \begin{align*}
    [\alpha](F(f_{u,i}))&=F([\alpha](f_{u,i}))\\
    &= F(\alpha^{k-2-i}\cdot\sigma(\alpha)^i\cdot f_{u,i}) \\
    &= \sigma(\alpha)^{k-2-i}\alpha^iF(f_{u,i}).
    \end{align*}
since $\sigma^2(\alpha)=\alpha.$ Therefore, $F(f_{u,i})$ is an eigenvector of $[\alpha]$ with the same eigenvalue as $f_{u,\tilde{\sigma}(i)}.$ Therefore, there exists $c_{u,i}\in K$ with 
    $$
    F(f_i) = p\cdot c_{u,i} f_{u,\tilde{\sigma}(i)}.
    $$
    Moreover, by Theorem 1 of \cite{MazurFrobHodgeFil} and Lemma~\ref{lem:quasiHodge}, we have that $\ord_p(c_{u,i})\geq k-2-i,$ and this gives our claim.
\end{proof}

The relations for the $c_{u,i}$ are as follows.

\begin{lemma}\label{lem:indDecompOrd}
    In the notation of Lemma~\ref{lem:indDecomp}, if $[K_p:\Q_p]=r$ and $0\leq i\leq \frac{k-2}{2},$ then we have that
    \begin{equation}\label{eq:genEigenFormula}
    d_{u,i} +\frac{p^r}{d_{u,i}}=p^{ri}\cdot\left(p^{r(k-2-2i)}+1-\#E(\F_{p^{r(k-2-2i)}})\right),
    \end{equation}
    where 
    $$
    d_{u,i}:=\sigma^{r-1}\left(c_{u,i}^{\lfloor(r+1)/2\rfloor}\right)\cdot\sigma^r(c_{u,\tilde{\sigma}(i)}^{\lfloor r/2\rfloor}).
    $$
    Moreover, if $E^{\univ}_{\kappa(u)}$ has ordinary reduction, then $\ord_p(c_{u,i})=k-2-i.$
\end{lemma}

\begin{proof}
By Lemma~\ref{lem:ladicChar}, we have that $F^r$ is a linear map with eigenvalues $\{\alpha^i\beta^{k-2-i} | 0\leq i\leq k-2\}$ such that 
$$
\alpha + \beta = p^r+1-\#E^{\univ}_{\kappa(u)}(\F_{p^r})\quad\mathrm{and}\quad\alpha\cdot\beta=p^r.
$$
Moreover, if $E^{\univ}_{\kappa(u)}$ has ordinary reduction then $\tilde{\sigma}(i)=i$ whereas if it has supersingular reduction, then $r$ is even, and $\tilde{\sigma}^r(i)=i.$ Therefore, in both cases, $f_{u,i}$ is an eigenvector of $F^r.$ By repeated applications of Lemma~\ref{lem:indDecomp} and noting that $F^r$ is linear, we have (\ref{eq:genEigenFormula}). If $E^{\univ}_{\kappa(u)}$ has ordinary reduction, then we have $\ord_p(d_{u,i})=r\ord_p(c_{u,i}).$ Moreover, Theorem 1 of \cite{MazurFrobHodgeFil} tells us that either $\alpha$ or $\beta$ is a $p$-adic unit, and since $\ord_p(c_{u,i})\geq k-2-i,$ we have our claim.
\end{proof}

\begin{remark}\label{rem:highDim}
    Note that if $\dim S_k(X)>0,$ then $\O$ acts on the space $\H^1(X,\Omega^\bullet)\backslash\H^1(X,\Omega_u^\bullet)$ instead of $\H^1(X,\Omega^\bullet_u).$ Therefore, in general, we cannot expect that an analogue of Corollary~\ref{cor:ASD-basis} exists when $\dim S_k(X)>0.$ On the other hand, in some situations, one has a global action on $\H^1(X,\Omega^\bullet_u).$ For instance, let $u=\frac{1}{2}$ and consider the situation of Theorem~\ref{thm:Gamma1(4)Basis}. In this case, the Atkin--Lehner involution $W_4,$ which acts by $\lambda\to 1-\lambda$ induces a global action on $\H^1(X,\Omega^\bullet_u).$  
    The same arguments show that  while $S_5(X_1(4))$ is one-dimensional,  the coefficients of the weight-5 meromorphic form 
    $$
    f(\tau)=\frac{1}{2\pi i}\frac{\theta^6\l'(\tau)}{(1-2\l(\tau))}=:\sum\limits_{n\geq 1} a_f(n)q^n
    $$ 
    satisfy
    $$
    a_f(mp^s)\equiv \mu_p^3 a_f(mp^{s-1})\pmod {p^{4s}}
    $$
    for all $p\equiv 1\pmod 4$ and $m,s\ge 1,$ where $\mu_p$ is a $p$-adic unit root.
   
\end{remark}

\section{Explicit Residue computations}\label{sec:ERC}

In order to deduce Theorems~\ref{thm:Gamma1(4)Basis} and \ref{thm:SL2(Z)Basis} from Corollary~\ref{cor:ASD-basis}, we require an explicit computation of the residue map in order to diagonalize with respect to the CM structure. For this calculation, it suffices to work analytically (see Sections 5.2 and 5.3 of \cite{FonsecaBrown}). 

To this end, we recall a result of Brown and Fonseca which computes $\Res$ in the Betti frame.

\begin{lemma}\label{lem:FBResidue}\cite[Lemma 5.10]{FonsecaBrown}
    Let $\Gamma$ be a finite-index subgroup of $\SL_2(\Z)$ and $X$ the corresponding modular curve $\Gamma\backslash\overline{\mathfrak{H}}.$ Moreover, let $f\in M_k(X,\star u)^{\an}$ and denote by $\sum\limits_{n\gg-\infty}c_n(f;w)(\tau-w)^n$ the Laurent expansion of $f$ around some $w\in\mathfrak{H}$ that lies above $u.$ Then, we have that
    $$
    \Res_{u}(f) = \sum\limits_{j=0}^{k-2}(2\pi i )^j\cdot \binom{k-2}{j}c_{-j-1}(f;w)\left(\frac{dx}{y}\right)^{k-2-j}\left(\frac{xdx}{y}-\frac{E_2(w)}{12}\cdot\frac{dx}{y}\right)^j,
    $$
    where $\frac{dx}{y},\frac{xdx}{y}$ are the basis in the de Rham cohomology $H^1_{dR}(E_u/\C)$ of the elliptic curve
    $$
    E_u:\ \ \ \ y^2=4x^3-\frac{E_4(w)}{12}x+\frac{E_6(w)}{216},
    $$
    and where $E_n(\tau)$ denotes the normalized Eisenstein series of weight $n$ for $\text{SL}_2(\Z)$.
\end{lemma}

\begin{remark}
    For finite index subgroup $\G$ of $\text{SL}_2(\Z),$ the space of quasi-modular forms for $\G$ is a finitely generated differential algebra closed under $\partial_{\tau}$ (see   \cite{zagier2000modular}). When $X_\G$ has genus 0 with a fixed choice of Hauptmodul $t$, expressing $E_4(w)$ and $E_6(w)$ using modular forms for $\G$ allows us to express $E_u$ in the previous result as a one-parameter family of elliptic curves over $\overline \Q(t)$. See \cite[Appendix]{HMM1} for a list of Hauptmoduln for $\G$ associated to arithmetic triangle groups, such as $\SL_2(\Z)$ and $\G_1(4)$, which have  only three cusps and elliptic points, as well as the logarithmic derivatives of these Hauptmoduln. 
\end{remark}

\subsection{\texorpdfstring{The case of $\Gamma_1(4)$}{The case of Gamma1(4)}}\label{subsec:ERCGamma1(4)}

In order to prove Theorem~\ref{thm:Gamma1(4)Basis}, we consider modular forms on $\Gamma_1(4)$ as $\Gamma_1(4)\backslash\Gamma_1(8)$-invariant modular forms on $X_1(8)$ (for example, see Section 6 of \cite{KS16}). Lemma~\ref{lem:FBResidue} allows us to compute the residue of the modular forms in Theorem~\ref{thm:Gamma1(4)Basis}. On the other hand, we can express $E_4(\tau), E_6(\tau),$ and the Laurent coefficients explicitly in terms of the values of the modular forms $\theta^2(\tau)$ and $\lambda(\tau)$ defined in Example~\ref{ex:ex1}. This is obtained by working on the differential algebra 
$\C[\theta^2(\tau),\lambda(\tau), E_2(\tau)]$\footnote{For our purposes, this is the appropriate analogue of $\C[E_2(\tau),E_4(\tau),E_6(\tau),j(\tau)]\cong\bigoplus\limits_{n\geq 2}M_n^{\text{quasi},!}(\SL_2(\Z),\C).$} where the differential operator is $\frac{1}{2\pi i}\frac{d}{d\tau}.$ Putting this together, we obtain an algebraic description of $\Res$ for a model of a ``universal'' elliptic curve.

We first compute $E_4(\tau)$ and $E_6(\tau)$ in terms of $\theta^2(\tau)$ and $\lambda(\tau).$ 

\begin{lemma}\label{lem:E4E6Gamma1(4)}
    If $\tau\in\mathfrak{H},$ then the following are true.
    \begin{enumerate}
        \item $$
        E_4(\tau)=(1+14\lambda(\tau)+\lambda(\tau)^2)\cdot \theta^8(\tau).
        $$
        \item $$
        E_6(\tau)=(1-33\lambda(\tau)-33\lambda(\tau)^2+\lambda(\tau)^3)\cdot\theta^{12}(\tau).
        $$
    \end{enumerate}
\end{lemma}

\begin{proof}
    This follows immediately from computing the first few Fourier coefficients and Sturm's bound (see for example Theorem 5.6.11 and Definition 5.6.13 of \cite{Cohen-Stromberg}).
\end{proof}

\subsubsection{Weight $3$ Computations}

For weight $3,$ we require the residues of $f(\tau):=\frac{\theta^2(\tau)}{(\lambda(\tau)-u)}\cdot \lambda'(\tau)$ and $g(\tau):=\frac{\theta^2(\tau)}{(\lambda(\tau)-u)^2}\cdot \lambda'(\tau),$ where $u=\lambda(w)\in\C$ and where $t'$ denotes $\frac{dt}{d\tau}$ for a function $t$ on $\mathfrak{H}.$ The required Laurent coefficients of $f(\tau)$ are as follows.

\begin{lemma}\label{lem:Resf31}
    If $w\in\mathfrak{H}$ and $f(\tau)=\sum\limits_{n\geq -1}c_n(f; w)(\tau-w)^n,$ then we have
    $$
    c_{-1}(f;w)=\theta^2(w).
    $$
\end{lemma}

\begin{proof}
    We have 
    $$
    c_{-1}(f;w)=\lim\limits_{\tau\to w}(\tau-w)\cdot f(\tau)=\lim\limits_{\tau\to w}\frac{\theta^2(\tau)}{(\lambda(\tau)-u)/(\tau-w)}\cdot \lambda'(\tau)=\theta^2(w).
    $$
\end{proof}

In order to compute $c_{-1}(g;w)$ and $c_{-2}(g;w),$ we make use of the following intermediate lemma.

\begin{lemma}\label{lem:Theta2Der}
    The following are true.
    \begin{enumerate}
    \item 
    $$
  \frac{1}{2\pi i}\frac{d\lambda}{d\tau}=\lambda(\tau)(1-\lambda(\tau))\theta^4(\tau).
    $$
    \item If $s\geq 1,$ then we have 
    $$
    \frac{1}{2\pi i}\frac{d\theta^{2s}(\tau)}{d\tau}=\frac{(5\lambda(\tau)-1)s}{12}\theta^{2s+4}(\tau)+\frac{s}{12}\theta^{2s}(\tau)\cdot E_2(\tau).
    $$
    \end{enumerate}
\end{lemma}

\begin{proof}
    (1) is Lemma 2.1 of \cite{LLT}. To prove (2), note that the Serre derivative of $\theta^2(\tau)$
    $$
    \frac{1}{2\pi i}\frac{d\theta^2(\tau)}{d\tau} - \frac{1}{12}\theta^2(\tau)\cdot E_2(\tau)
    $$
    is a holomorphic modular form on $\Gamma_1(4)$ (see Proposition 5.3.6 (b) of \cite{Cohen-Stromberg}). The claim for $s=1$ then follows from computing the first few Fourier coefficients and Sturm's bound. The claim for $s>1$ is just the chain rule.
\end{proof}

We now compute the required Laurent coefficients of $g(\tau)=\frac{\theta^2(\tau)}{(\lambda(\tau)-u)^2}\cdot \lambda'(\tau).$

\begin{lemma}\label{lem:Resf32}
    If $w\in\mathfrak{H}$ and $g(\tau)=\sum\limits_{n\geq -2}c_n(g;w)(\tau-w)^n,$ then we have
    $$
    c_{-1}(g;w) = \frac{(5u-1)\theta^2(w)}{12u(1-u)}+\frac{E_2(w)}{12u(1-u)\theta^2(w)}
    $$
    and
    $$
    c_{-2}(g;w)=\frac{1}{2\pi i}\cdot \frac{1}{u(1-u)\theta^2(w)}.
    $$
\end{lemma}

\begin{proof}
    We first compute $c_{-2}(g;w).$ We have
    $$
    c_{-2}(g;w)=\lim\limits_{\tau\to w}\frac{\theta^2(\tau)}{(\lambda(\tau)-u)^2/(\tau-w)^2}\cdot\lambda'(\tau)=\frac{\theta^2(w)}{\lambda'(w)}.
    $$
    The claim then follows from (1) of Lemma~\ref{lem:Theta2Der}. 
    
    To prove (2), write $k(\tau)=\theta^2(\tau)\cdot\frac{d\lambda}{d\tau}.$ We have
    \begin{align*}
        g(\tau)&=\frac{k(\tau)}{(\lambda(\tau)-u)^{2}} = \left[\sum\limits_{n\geq 0}c_n(k;w)(\tau - w)^n\right]\cdot\frac{1}{\left[\sum\limits_{n\geq 1}c_n(\lambda;w)(\tau-w)^n\right]^2} \\
        &=\left[\sum\limits_{n\geq 0}c_n(k;w)(\tau - w)^n\right]\cdot\frac{1}{c_1(\lambda;w)^2}\cdot\frac{1}{(\tau-w)^2}\cdot\left[\frac{1}{1+\sum\limits_{n\geq 2}\frac{c_n(\lambda;w)}{c_1(\lambda;w)}\cdot(\tau-w)^{n-1}}\right]^2 \\
        &=\left[\sum\limits_{n\geq 0}c_n(k;w)(\tau - w)^n\right]\cdot\frac{1}{c_n(\lambda;w)^2}\cdot\frac{1}{(\tau-w)^2} \\
        &\ \ \cdot\left[1-2\sum\limits_{n\geq 2}\frac{c_n(\lambda;w)}{c_1(\lambda;w)}(\tau-w)^{n-1}+O((\tau-w)^2)\right].
    \end{align*}
    Therefore, we have
    $$
    c_{-1}(g;w)=\frac{1}{c_1(\lambda;w)^2}\left[-2c_0(k;w)\cdot\frac{c_2(\lambda;w)}{c_1(\lambda;w)}+c_1(k;w) \right].
    $$
    Now, we have $c_0(k;w)=k(w)=\theta^2(w)\cdot\lambda'(w).$ Moreover, we have
    $$
    \frac{c_2(\lambda;w)}{c_1(\lambda;w)}=\frac{1}{2}\cdot\frac{\lambda''(w)}{\lambda'(w)}
    $$ 
    and therefore
    $$
    -2c_0(k;w)\frac{c_2(\lambda;w)}{c_1(\lambda;w)}=-\theta^2(w)\lambda''(w).
    $$
    On the other hand, the product rule gives us that
    $$
    c_1(k;w)=\frac{d\left[\theta^2(\tau)\cdot\lambda'(\tau)\right]}{d\tau}|_{\tau = w}=\theta^2(w)\lambda''(w)+\frac{d\theta^2}{d\tau}|_{\tau = w}\cdot\lambda'(w).
    $$
    Therefore, the claim follows from (2) of Lemma~\ref{lem:Theta2Der}.
\end{proof}

We put this together to compute the modular forms of weight $3$ with appropriate residues. 

\begin{lemma}\label{lem:Wt3Gamma1(4)Res}
    If $u\in\C\setminus\{0,1\},$ consider the elliptic curve
    $$
    y^2=4x^3-\frac{(1+14u+u^2)}{12}x+\frac{1-33u-33u^2+u^3}{216}.
    $$
    In this notation, the following are true.
    \begin{enumerate}
        \item $\displaystyle \Res\left(\frac{\theta^2(\tau)}{(\lambda(\tau)-u)}\cdot\lambda'(\tau)\right)=\frac{dx}{y}.$
        \item $$
        \Res\left(u(1-u)\cdot\frac{\theta^2(\tau)}{(\lambda(\tau)-u)^2}\cdot\lambda'(\tau)-\frac{5u-1}{12}\cdot\frac{\theta^2(\tau)}{(\lambda(\tau)-u)}\lambda'(\tau)\right)=\frac{xdx}{y}.
        $$
    \end{enumerate}
\end{lemma}

\begin{proof}
    This follows immediately from Lemmas~\ref{lem:FBResidue},\ref{lem:E4E6Gamma1(4)}, \ref{lem:Resf31},\ref{lem:Resf32}, and normalizing by $x\to\theta^4(w)x, y\to\theta^6(w)y.$ 
\end{proof}

\subsubsection{Weight $4$ Computations}

For weight $4,$ we require the residues of $$f(\tau):=\frac{\theta^4(\tau)}{(\lambda(\tau)-u)}\cdot \lambda'(\tau),\quad g(\tau):=\frac{\theta^4(\tau)}{(\lambda(\tau)-u)^2}\cdot \lambda'(\tau),\quad \text{and} \quad h(\tau):=\frac{\theta^4(\tau)}{(\lambda(\tau)-u)^3}\cdot\lambda'(\tau),$$ where $u=\lambda(w)\in\C.$ The required Laurent coefficients of $f(\tau)$ are as follows.

\begin{lemma}\label{lem:Resf41}
    If $w\in\mathfrak{H}$ and $f(\tau)=\sum\limits_{n\geq -1}c_n(f;w)(\tau-w)^n,$ then we have
    $$
    c_{-1}(f;w)=\theta^4(w).
    $$
\end{lemma}

\begin{proof}
    The proof is analogous to that of Lemma~\ref{lem:Resf31}.
\end{proof}

We now compute the required Laurent coefficients of $g(\tau).$

\begin{lemma}\label{lem:Resf42}
    If $w\in\mathfrak{H}$ and $g(\tau)=\sum\limits_{n\geq -2}c_n(g;w)(\tau-w)^n,$ then we have
    $$
    c_{-1}(g;w)=\frac{(5u-1)\theta^4(w)}{6u(u-1)}+\frac{E_2(w)}{6u(1-u)}
    $$
    and
    $$
    c_{-2}(g;w)=\frac{1}{2\pi i}\cdot \frac{1}{u(1-u)}.
    $$
\end{lemma}

\begin{proof}
    The computation of $c_{-2}(g;w)$ is analogous to that in Lemma~\ref{lem:Resf32}. In order to compute $c_{-1}(g;w),$ let $k(\tau):=\theta^4(\tau)\cdot\frac{d\lambda}{d\tau}=:\sum\limits_{n\geq 0}c_n(k;w)(\tau-w)^n$ and write $\lambda(\tau)=:\sum\limits_{n\geq 0}c_n(\lambda;w)(\tau-w)^n.$ Then, as in the proof of Lemma~\ref{lem:Resf42}, we have
    $$
    c_{-1}(g;w)=\frac{1}{u(1-u)\theta^4(w)}\cdot\frac{d\theta^4}{d\tau}|_{\tau=w}.
    $$
    The claim then follows from (2) of Lemma~\ref{lem:Theta2Der}.
\end{proof}

In order to compute $c_{-1}(h;w),c_{-2}(h;w),$ and $c_{-3}(h;w),$ we make use of the following lemma.

\begin{lemma}\label{lem:SomeExtraComps} 
    The following are true.

    \begin{enumerate}
        \item $\displaystyle \frac{1}{2\pi i}\lambda''(\tau) = \frac{\lambda'} {6}\left((5-7\lambda)\theta^4+E_2\right).$ 
        \item $\displaystyle \frac{1}{2\pi i}\frac{dE_2}{d\tau}=\frac{E_2^2}{12}-\frac{(1+14\lambda+\lambda^2)\theta^8}{12}.$ 
        \item $$
        \frac{1}{(2\pi i)^2}\lambda'''=\frac{\lambda'}{6}\left[\frac{37\lambda^2-58\lambda+13}{4}\theta^{8}+\frac{5-7\lambda}{2}\cdot\theta^4E_2+\frac{1}{4}E_2^2\right]
        $$
    \end{enumerate}
\end{lemma}

\begin{proof}
    (1) follows from the product rule and Lemma~\ref{lem:Theta2Der}.

    To show (2), note that we have (see Proposition 5.3.10 of \cite{Cohen-Stromberg}) 
    $$
    \frac{1}{2\pi i}\frac{dE_2}{d\tau}=\frac{E_2^2-E_4}{12}
    $$
    and the claim follows by (1) of Lemma~\ref{lem:E4E6Gamma1(4)}.

    To prove (3), note that (1) of Lemma~\ref{lem:Theta2Der} and the product rule applied to (1) gives
    \begin{align*}
    \frac{1}{(2\pi i)^2}\lambda'''(\tau)&=\lambda'(\tau)\left[\frac{21\lambda^2-24\lambda+5}{6}\theta^8+\frac{1-2\lambda}{6}\theta^4\cdot E_2\right]+\frac{\lambda(1-\lambda)(5-7\lambda)}{6}\cdot\frac{d\theta^8}{d\tau}\\ &+\frac{\lambda(1-\lambda)}{6}\cdot\frac{d(\theta^4 E_2)}{d\tau}.
    \end{align*}
    By (2) of Lemma~\ref{lem:Theta2Der}, we have that
    $$
    \frac{\lambda(1-\lambda)(5-7\lambda)}{6}\frac{d\theta^8}{d\tau} = \frac{\lambda'\cdot(5-7\lambda)}{6}\left[\frac{5\lambda-1}{3}\theta^8+\frac{1}{3}\theta^4E_2\right].
    $$
    Similarly, we have that 
    $$
    \frac{\lambda(1-\lambda)}{6}\frac{d(\theta^4E_2)}{d\tau}=\frac{\lambda'}{6}\left(\frac{5\lambda-1}{6}\theta^4E_2+\frac{1}{4}E_2^2-\frac{1+14\lambda+\lambda^2}{12}\cdot\theta^8\right).
    $$
    Therefore, we conclude that
    $$
    \frac{1}{(2\pi i)^2}\lambda'''=\frac{\lambda'}{6}\left[\frac{37\lambda^2-58\lambda+13}{4}\theta^8+\frac{5-7\lambda}{2}\cdot\theta^4E_2+\frac{1}{4}E_2^2\right].
    $$
\end{proof}

We now compute the required Laurent coefficients of $h(\tau).$

\begin{lemma}\label{lem:Resf43}
    If $w\in\mathfrak{H}$ and $h(\tau)=\sum\limits_{n\geq -3}c_n(h;w)(\tau-w)^n,$ then we have the following formulas.
    \begin{enumerate}
    \item \begin{align*}
    c_{-1}(h;w)=\frac{1}{u^2(1-u)^2\theta^{12}(w)}&\Bigg[\frac{1}{144}(109u^2-58u+13)\theta^{16}(w)\\
    &+\frac{1}{72}(17u-7)\theta^{12}(w)E_2(w)+\frac{1}{144}\cdot \theta^8(w)E_2^2(w)\Bigg].
    \end{align*}
    
    \item $$
    c_{-2}(h;w)=\frac{1}{2\pi i}\cdot\left(\frac{17u-7}{12u^2(1-u)^2}+\frac{1}{12u^2(1-u)^2\theta^4(w)}E_2(w)\right).
    $$
    
    \item $$
    c_{-3}(h;w)=\frac{1}{(2\pi i)^2}\cdot\frac{1}{u^2(1-u)^2\theta^4(w)}.
    $$
    \end{enumerate}
\end{lemma}

\begin{proof}
    We write $k(\tau)=\theta^4(\tau)\cdot\frac{d\lambda}{d\tau}.$ We then have that 
    \begin{align*}
        h(\tau)&=\left[\sum\limits_{n\geq 0}c_n(k;w)(\tau-w)^3\right]\cdot\frac{1}{(\tau-w)^3}\cdot\frac{1}{c_1(\lambda;w)^3} \\
        &\ \ \cdot\left[1-3\sum\limits_{n\geq 2}\frac{c_n(\lambda;w)}{c_1(\lambda;w)}(\tau-w)^{n-1}+6\left(\sum\limits_{n\geq 2}\frac{c_n(\lambda;w)}{c_1(\lambda;w)}(\tau-w)^{n-1}\right)^2 +O((\tau-w)^3)\right] \\
        &=\left[\sum\limits_{n\geq 0}c_n(k;w)(\tau-w)^3\right]\cdot\frac{1}{(\tau-w)^3}\cdot\frac{1}{c_1(\lambda;w)^3}\\
        &\ \ \cdot\Bigg[1-3\frac{c_2(\lambda;w)}{c_1(\lambda;w)}(\tau-w)-3\frac{c_3(\lambda;w)}{c_1(\lambda;w)}(\tau-w)^2+6\left(\frac{c_2(\lambda;w)}{c_1(\lambda;w)}\right)^2\cdot(\tau-w)^2+O((\tau-w)^3)\Bigg].
    \end{align*}

    Therefore, it follows that we have
    $$
    c_{-3}(h;w) = \frac{c_0(k;w)}{c_1(\lambda;w)^3},
    $$

    $$
    c_{-2}(h;w)=\frac{1}{c_1(\lambda;w)^3}\left[-3\frac{c_0(k;w)c_2(\lambda;w)}{c_1(\lambda;w)}+c_1(k;w)\right],
    $$
    and
    $$
    c_{-1}(h;w)=\frac{1}{c_1(\lambda;w)^3}\left[c_0(k;w)\cdot\left(-3\frac{c_3(\lambda;w)}{c_1(\lambda;w)}+\frac{6c_2(\lambda;w)^2}{c_1(\lambda;w)^2}\right)-3\frac{c_1(k;w)c_2(\lambda;w)}{c_1(\lambda;w)}+c_2(k;w)\right].
    $$

    The formula for $c_{-3}(h;w)$ clearly follows from (2) of Lemma~\ref{lem:Theta2Der}. 

    To obtain the formula for $c_{-2}(h;w),$ note that we have
    $$
    -3\frac{c_0(k;w)c_2(\lambda;w)}{c_1(\lambda;w)}=-3\frac{\theta^4(w)\lambda'(w)\lambda''(w)}{2\lambda'(w)}=-\frac{3}{2}\theta^4(w)\lambda''(w).
    $$

    On the other hand, the product rule gives us that
    $$
    c_1(k;w)=k'(w)=\lambda'(w)\cdot\frac{d\theta^4}{d\tau}|_{\tau=w}+\theta^4(w)\lambda''(w).
    $$
    Therefore, we have that
    $$
    c_{-2}(h;w)=\frac{1}{\lambda'(w)^2}\left[-\frac{1}{2}\theta^4\frac{\lambda''(w)}{\lambda'(w)}+\frac{d\theta^4}{d\tau}|_{\tau=w}\right].
    $$
    Applying Lemma~\ref{lem:Theta2Der} and (1) of Lemma~\ref{lem:SomeExtraComps}, we have our claim.

    Finally, we compute $c_{-1}(h;w).$ By (3) of Lemma~\ref{lem:SomeExtraComps}, we have that
    $$
    -3\frac{c_3(\lambda;w)}{c_1(\lambda;w)}=-\frac{(2\pi i)^2}{12}\left[\frac{37u^2-58u+13}{4}\theta^8(w)+\frac{5-7u}{2}\cdot\theta^4(w)E_2(w)+\frac{1}{4}E_2^2(w)\right]
    $$
    On the other hand, by (1) of Lemma~\ref{lem:SomeExtraComps} we have
    $$
    6\frac{c_2(\lambda;w)^2}{c_1(\lambda;w)^2}=\frac{(2\pi i)^2}{24}\cdot\left[(5-7u)^2\theta(w)^8+2E_2(w)(5-7u)\theta(w)^4+E_2^2(w)\right]
    $$
    Therefore, we have 
    \begin{align*}
    \frac{1}{(2\pi i)^3}c_0(k;w)&\cdot\Bigg[-3\frac{c_3(\lambda;w)}{c_1(\lambda;w)}+\frac{6c_2(\lambda;w)^2}{c_1(\lambda;w)^2}\Bigg]= \frac{61u^2-82u+37}{48}\theta^{16}(w)u(1-u) \\
    &+u(1-u)\frac{5-7u}{24}\theta^{12}(w)E_2(w)+\frac{u(1-u)\theta^8(w)}{48}E_2^2(w).
    \end{align*}
    Moreover, by similar computations, we have
    \begin{align*}
    -3\cdot\frac{1}{(2\pi i)^3}\frac{c_1(k;w)c_2(\lambda;w)}{c_1(\lambda;w)}&=\frac{-u(1-u)(2-u)(5-7u)\theta^{16}(w)}{12}\\
    &-\frac{u(1-u)(7-8u)\theta^{12}(w)E_2(w)}{12}-\frac{u(1-u)}{12}\theta^8(w)E_2^2(w).
    \end{align*}
    We also similarly have
    \begin{align*}
    \frac{1}{(2\pi i)^3}\cdot c_2(k;w)&=\frac{\theta^{16}(w)}{72}u(1-u)(5u^2-20u+11)\\
    &+\frac{5\theta^{12}(w)E_2(w)}{36}u(1-u)(2-u)+\frac{5u(1-u)}{72}\theta^8(w)E_2^2(w).
    \end{align*}
    Putting this together with (1) of Lemma~\ref{lem:Theta2Der}, we have our claimed $c_{-1}(h;w).$
\end{proof}

We put Lemmas~\ref{lem:Resf41}, \ref{lem:Resf42}, and \ref{lem:Resf43} together to compute meromorphic modular forms of weight $4$ with appropriate residues. 

\begin{lemma}\label{lem:Wt4Gamma1(4)Res}
    If $u\in\C\setminus\{0,1\},$ consider the elliptic curve
    $$
    y^2=4x^3-\frac{(1+14u+u^2)}{12}x+\frac{1-33u-33u^2+u^3}{216}.
    $$
    In this notation, the following are true.
    \begin{enumerate}
        \item $\Res\left(\frac{\theta^4(\tau)}{(\lambda(\tau)-u)}\cdot\lambda'(\tau)\right)=\left(\frac{dx}{y}\right)^2.$
        \item $$
        \Res\left(\frac{u(1-u)\theta^4(\tau)\lambda'(\tau)}{2(\lambda(\tau)-u)^2}-\frac{(5u-1)\theta^4(\tau)\lambda'(\tau)}{12(\lambda(\tau)-u)}\right) = \frac{dx}{y}\cdot\frac{xdx}{y}.
        $$
        \item         
        $$
        \Res\left(\frac{u^2(1-u)^2\theta^4(\tau)\lambda'(\tau)}{(\lambda(\tau)-u)^3}+Q(u)\frac{\theta^4(\tau)\lambda'(\tau)}{(\lambda(\tau)-u)^2}+P(u)\frac{\theta^4(\tau)\lambda'(\tau)}{(\lambda(\tau)-u)}\right) = \left(\frac{xdx}{y}\right)^2,
        $$
        where $\displaystyle P(u):=\frac{1}{144}(61u^2-46u+1)$ and $\displaystyle Q(u):=\frac{-u(1-u)(17u-7)}{12}.$
    \end{enumerate}
\end{lemma}

\begin{remark}\label{rem:GMHyper}

As an alternative to this computation, one can use the proof of Proposition~\ref{prop:quasi} and the Gauss--Manin connection to compute these residues. For example, the Legendre family of elliptic curves 
$$
E^{\Leg}_\lambda: y^2=x(1-x)(1-\l x)
$$  
parameterized by $\lambda\in X(2)$ is $2$-isogenous to the $E_u$ used above. In this setting, for $0<\lambda<1,$ a real period can be given as 
$$
\int_0^1 \frac{dx}{y}=\int_0^1\frac{dx}{\sqrt{x(1-x)(1-\l x)}}=\pi\pFq21{\frac12&\frac12}{&1}{\l},
$$
where ${_2F_1}$ is the classical hypergeometric function, and a similar expression exists for the quasi-periods. Moreover, locally near the cusp $i\infty,$ we have (for example, see \cite{BB} and \cite{Yang-modular})
$$
\theta(\tau)^2=\pFq21{\frac12&\frac12}{&1}{\l(\tau)}
$$ 
Using this setting, the residue computations reduce to hypergeometric identities.
\end{remark}

\subsection{\texorpdfstring{The case of $\SL_2(\Z)$}{The case of SL2(Z)}}\label{subsec:ERCSL2(Z)}

In order to deduce Theorem~\ref{thm:SL2(Z)Basis} from Corollary~\ref{cor:ASD-basis}, we compute the residue associated to the form $f_{k,r,\alpha_D}$ in \eqref{eq:f(k,r,z)}, where $\displaystyle \alpha_D:=\frac{-b+\sqrt{-D}}{2a}\in\mathfrak{H}$ is an imaginary quadratic number.  We then show that the residue is an eigenvector of the CM action on $H^1_{dR}(E_{\alpha_D}/\C).$

In order to obtain the residue, we make use of the following elementary lemma.

\begin{lemma}\label{lem:EkE14-kRel}
    If $k\in\{4,6,8,10,14\},$ then we have
    $$
    2\pi i\frac{E_k(\tau)E_{14-k}(\tau)}{\Delta(\tau)j'(\tau)}=-1.
    $$
\end{lemma}

\begin{proof}
    This is clearly a holomorphic modular form of weight $0$ and therefore a constant. Equality follows from comparing the constant term of the $q$-series. 
\end{proof}

By Lemma~\ref{lem:EkE14-kRel}, we can write 
$$
f_{k,r,\alpha_D}(\tau) = -\frac{1}{2\pi i} E_k(\tau) R^r_{2-k,z}\left(\frac{1}{E_k(z)}\cdot\frac{j'(z)}{j(\tau)-j(z)}\right)\Bigg|_{z={\alpha_D}},
$$
where $R^r_{2-k,z}$ is defined in \eqref{eq:R(k,tau)}.
In order to compute the residue, we make use of the following identity which expresses $R_{2-k,z}^r$ in terms of $\frac{\partial}{\partial z}.$

\begin{lemma}\cite[Lemma 5.4]{HMFBook}\label{lem:RaisingOpId}
    If $k\in\Z$ and $r\geq 0,$ then we have
    $$
    R^r_{2-k,\tau}=\frac{1}{(4\pi)^r}\sum\limits_{l=0}^r(-1)^l\binom{r}{l}(2-k+l)_{r-l}\cdot y^{l-r}\left(-2i\frac{\partial}{\partial\tau}\right)^l,
    $$
    where $\tau=x+iy$ and $(x)_n$ is the rising factorial
    $$
    (x)_n:=\begin{cases}
        1 & \text{ if }n=0, \\
        x(x+1)\ldots(x+n-1) & \text{ if }n>0.
    \end{cases}
    $$
\end{lemma}

To compute the Laurent series of $f_{k,r,\alpha_D}(\tau)$ at $\tau=\alpha_D,$ we make use of the following lemma.

\begin{lemma}\label{lem:LaurentSeries}
Fix $\tau_0\in\mathfrak{H}$. Let $f$ and $g$ be holomorphic functions at $\tau_0$ with $g'(\tau_0)\neq 0.$ If $r\geq 0$ and
$$
h(\tau):=f(\tau)\frac{\partial^r}{\partial z^r}\Bigg|_{z=\tau_0}\left(\frac{1}{f(z)}\cdot\frac{g'(z)}{g(z)-g(\tau)}\right),
$$
then we have
$$
h(\tau)=r!(\tau-\tau_0)^{-r-1}+\sum\limits_{n\geq 0}c_n(\tau-\tau_0)^n.
$$
\end{lemma}

\begin{proof}
    The statement is equivalent to showing that
    $$
    \frac{\partial^r}{\partial z^r}\Bigg|_{z=\tau_0}\left(\frac{1}{f(z)}\cdot\frac{g'(z)}{g(z)-g(\tau)}-\frac{1}{f(\tau)(z-\tau)}\right)
    $$
    is holomorphic at $\tau=\tau_0.$
    Since $f$ and $g$ are analytic near $\tau_0,$ we have that 
    $$
    \frac{1}{f(z)}\cdot\frac{g'(z)}{g(z)-g(\tau)}-\frac{1}{f(\tau)(z-\tau)}
    $$
    is analytic in $(z,\tau)$ near $(\tau_0,\tau_0).$ Therefore, we have
    \begin{multline*}
    \lim\limits_{\tau\to\tau_0}\frac{\partial^r}{\partial z^r}\Bigg|_{z=\tau_0}\left(\frac{1}{f(z)}\cdot\frac{g'(z)}{g(z)-g(\tau)}-\frac{1}{f(\tau)(z-\tau)}\right) \\ = \frac{\partial^r}{\partial z^r}\Bigg|_{z=\tau_0}\left(\lim\limits_{\tau\to\tau_0}\left(\frac{1}{f(z)}\cdot\frac{g'(z)}{g(z)-g(\tau)}-\frac{1}{f(\tau)(z-\tau)}\right)\right)
    \end{multline*}
    and this is clearly holomorphic.
\end{proof}

We now compute the residue of $f_{k,r,\alpha_D}(\tau)$.

\begin{lemma}\label{lem:ResSL2(Z)}
    If $k\in\{4,6,8,10,14\}$, then we have\footnote{The notation $\doteq$ denotes equality up to a multiplicative constant.}
    $$
    \Res(f_{k,r,\alpha_D})\doteq \left(\frac{1}{\Im(\alpha_D)}-\left(\frac{1}{\pi}\mathlarger{\wp}_{\alpha_D}(w)+\frac{\pi}{3}E_2(w)\right)\right)dw^{k-2},
    $$
    where $w$ is the coordinate on the elliptic curve $E_{\alpha_D}/\C\cong\frac{\C}{\Z+\Z\alpha_D}$ and $\wp_{\alpha_D}$ is the Weierstrass elliptic function with periods $1$ and $\alpha_D$.
\end{lemma}

\begin{proof}
    By Lemmas~\ref{lem:RaisingOpId} and \ref{lem:LaurentSeries}, if $f_{k,r,\alpha_D}(\tau)=\sum\limits_{n\geq -r-1} c_n(f;w)(\tau-\alpha_D)^n$ and $0\leq m\leq r,$ we have that
    $$
    c_{-m-1}(f_{k,r,\alpha_D})=\frac{-1}{2\pi i}\cdot\frac{1}{(2\pi)^r}\cdot\frac{r!}{(r-m)!}(2-k+m)_{r-m}(2\Im(\alpha_D))^{m-r}\cdot i^m.
    $$
    By Lemma~\ref{lem:FBResidue}, we have that
    $$
    \Res(f_{k,r,\alpha_D})\doteq \left(\sum\limits_{m=0}^r\binom{k-2}{m}\frac{(2-k+m)_{r-m}}{(r-m)!}(2\Im(\alpha_D))^{m-r}\cdot\left(\frac{1}{2\pi}\mathlarger{\wp}_{\alpha_D}(w)+\frac{\pi}{6}E_2(w)\right)^m\right)dw^{k-2}.
    $$
    Moreover, we have that
    $$
    (2-k+m)_{r-m} = (-1)^{r-m}\cdot\frac{(k-2-m)!}{(k-2-r)!},
    $$
    and then the binomial theorem gives our claim.
\end{proof}

Finally, we show that this residue is an eigenvector of CM on $\Sym^{k-2}H^1_{dR}(E_{\alpha_D}/\C).$ This clearly follows from the following lemma.

\begin{lemma}\label{lem:Exactness}
    If $\alpha_D:=\frac{-b+\sqrt{-D}}{2a}\in\mathfrak{H}$ is a CM point and 
    $$
    \delta(w):=\left(\frac{1}{\Im(\alpha_D)}-\frac{1}{\pi}\mathlarger{\wp}_{\alpha_D}(w)-\frac{\pi}{3}E_2(w)\right)dw,
    $$
    then $\delta(a\alpha_Dw)-a\overline{\alpha_D}\cdot\delta(w)$
    is an exact differential form.
\end{lemma}

\begin{proof}

    Write $\Lambda:=\Z\oplus\Z\alpha_D.$ First, note that $\delta(a\alpha_Dw)$ is well-defined on $\C/\Lambda$ since $a\alpha_D\in\Lambda$ and $a\alpha_D\cdot\alpha_D=-b\alpha_D-c\in\Lambda.$
    To show that $\delta(a\alpha_Dw)-a\overline{\alpha_D}\delta(w)$ is an exact differential form, note that $\delta(w)$ has zero residues.    
    Therefore, fix $t_0\in\C\setminus\Lambda,$ and define
    $$
    G(t)=\int_{\gamma_t}\delta(a\alpha_Dw)-a\overline{\alpha_D}\delta(w)dw,
    $$
    where $\gamma_t$ is a path in $\C\setminus\Lambda$ from $t_0$ to $t.$ By the residue theorem, this is well-defined and is independent of the choice of the path $\gamma_t.$ 
    
    It remains to show that $G(t)$ is elliptic with respect to the lattice $\Lambda.$ This is equivalent to 
    $$\int_0^1\delta(a\alpha_Dw)-a\overline{\alpha_D}\delta(w)dw=\int_0^{\alpha_D}\delta(a\alpha_Dw)-a\overline{\alpha_D}\delta(w)dw =0.
    $$
    We have that (for example, see Lemma A1.3.9 of \cite{Katz-padic}) 
    $$
    \int_0^1\wp_{\alpha_D}(w)dw=-\frac{\pi^2}{3}E_2(\alpha_D).
    $$
    Moreover, the Legendre relation (for example, see A1.3.4 of \cite{Katz-padic}) gives us that
    $$
    \int_0^1\wp_{\alpha_D}(\alpha_Dw)d(\alpha_Dw)=\int_0^{\alpha_D}\wp_{\alpha_D}(w)dw = 2\pi i +\alpha_D\int_0^1\wp_{\alpha_D}(w)dw.
    $$
    Putting all this together, we have our claim.
\end{proof}

\section{Proofs of Theorems~\ref{thm:main-ASD},\ref{thm:Gamma1(4)Basis}, \ref{thm:SL2(Z)Basis} and Corollary~\ref{cor:ASD-basis}}\label{sec:proofs}

\begin{proof}[Proof of Theorem~\ref{thm:main-ASD}]
    Proposition~\ref{prop:main-ASD-weak} gives the recurrence relation and Corollary~\ref{cor:jfp1} and (3) of Proposition~\ref{prop:extraP} give the formula for $j_{f,p}.$
\end{proof}

\begin{proof}[Proof of Corollary~\ref{cor:ASD-basis}]
    This follows from Theorem~\ref{thm:main-ASD}, Lemma~\ref{lem:indDecomp}, and Lemma~\ref{lem:indDecompOrd}.
\end{proof}

\begin{proof}[Proof of Theorem~\ref{thm:Gamma1(4)Basis}]
    This follows from the proof of Corollary~\ref{cor:ASD-basis}, Lemma~\ref{lem:Wt3Gamma1(4)Res}, Lemma~\ref{lem:Wt4Gamma1(4)Res}, 1.3 of \cite{Sch88}, and Theorem 1 of \cite{HLLT}. 
\end{proof}

\begin{proof}[Proof of Theorem~\ref{thm:SL2(Z)Basis}]
    This follows from the proof of Corollary~\ref{cor:ASD-basis}, Lemma~\ref{lem:ResSL2(Z)}, Lemma~\ref{lem:Exactness}, 1.3 of \cite{Sch88}, and Theorem 1 of \cite{HLLT}.
\end{proof}

\section{Examples and Related Discussions}\label{sec:ExRD}

In this section, we discuss some of the material in the previous sections.

\subsection{\texorpdfstring{Diagonalizing $H^1_{dR}$ for CM elliptic curves}{Diagonalizing de Rham cohomology for CM elliptic curves}}\label{ss:Diagonalizing-CM-EC}

Here we explain how to obtain the $c_1$ and $c_2$ in Theorem~\ref{thm:Gamma1(4)Basis} for a given CM $E_u$. Besides the algorithm we provide, they can be computed using singular values of modular forms and quasi-modular forms as in \cite{BBRSTT} which applies when $X$ is a modular curve or a Shimura curve (see the set up by Yang in \cite{Yang-Schwarzian}). We have implemented the algorithm given below in \texttt{Sagemath}. Moreover, we note that this is related to a classic problem of finding Ramanujan-type formulas for $\frac{1}{\pi}$ originating in the work of Chudnovsky--Chudnovsky \cite{Chu-Chu} and Borwein--Borwein \cite{BB} (see \cite{Zudilin-2nd} for an overview of this).  We provide an example consequence of this algorithm.

For simplicity, in what follows, we suppose that $K$ is an algebraically closed field of characteristic $0.$ The de Rham cohomology of an elliptic curve $E/K$ in Weierstrass form
$$
E:\ \ \ Y^2 = X^3 + AX + B
$$
can be described as the $K$-vector space spanned by $\frac{dX}{Y}$ and $\frac{XdX}{Y}$ modulo exact forms (for example, see A1.2.3 of \cite{Katz-padic}).  

We first determine the effect of an isogeny on $\frac{XdX}{Y}.$

\begin{prop}\label{prop:IsogenyActionDeRham}
    Consider two elliptic curves $E/K,E'/K$ defined by the equations
    $$
    E:\ \ \ y^2=x^3+Ax+B
    $$
    and
    $$
    E':\ \ \ Y^2=X^3+\tilde A X+\tilde B.
    $$
    If $\phi:E\to E'$ is an isogeny with $\phi^\ast(\frac{dX}{Y}) = \alpha_\phi\frac{dx}{y},$ then we have
    $$
    \phi^\ast\left(\frac{XdX}{Y}\right) = \frac{\deg\phi}{\alpha_\phi}\frac{xdx}{y}-\frac{1}{\alpha_\phi}\left(\sum\limits_{P\in\ker\phi\setminus\{0\}}x(P)\right)\frac{dx}{y} + \text{ exact form}.
    $$
\end{prop}

To prove Proposition~\ref{prop:IsogenyActionDeRham}, we use the following well-known lemma to decompose isogenies into the three ``basic'' types.

\begin{lemma}\label{lem:IsogenyDecomposition}
    Every isogeny can be written as a composition of isomorphisms, isogenies with cyclic kernels (cyclic isogenies), and multiplication by integers.
\end{lemma}
\begin{proof}
    This follows directly from Theorem III.4.10 and Proposition III.4.12 of \cite{Silverman} and from the structure theorem for finite abelian groups.
\end{proof}

We now prove Proposition~\ref{prop:IsogenyActionDeRham} for each of the basic types.

\begin{lemma}\label{lem:IsomorphismActionDeRham}
    Proposition~\ref{prop:IsogenyActionDeRham} is true when $\phi$ is an isomorphism.
\end{lemma}

\begin{proof}
    In this situation, there exists $u\in K^\times$ such that $X=u^2x, Y=u^3y,$  and $\phi(x,y)=(X,Y).$ The formula follows trivially with $\alpha_\phi=\frac 1u.$
\end{proof}

\begin{lemma}\label{lem:cyclicActionDeRham}
    Proposition~\ref{prop:IsogenyActionDeRham} is true when $\phi$ is cyclic.
\end{lemma}

\begin{proof}
    $\phi$ is unique up to isomorphism of the codomain $E'.$ Therefore, it suffices to consider the situation (see \cite{VeluFormulas}) with
    \begin{enumerate}
        \item $\ker\phi =: G$ is cyclic.
        \item For $Q\in G, Q=:(x_Q,y_Q),$ define $t_Q:=3x_Q^2+A, u_Q:=2y_Q^2, w_Q:=u_Q+t_Qx_Q.$
        \item Define $t:=\sum\limits_{Q\in G\setminus\{O\}}t_Q, w:=\sum\limits_{Q\in G\setminus\{O\}} w_Q,$
        where $O$ is the origin of the elliptic curve.
        \item Define
        $$
        r(x):= x+\sum\limits_{Q\in G\setminus\{O\}}\frac{t_Q}{x-x_Q}+\frac{u_Q}{(x-x_Q)^2}.
        $$
        \item Define $\tilde{A}:=A-5t, \tilde{B}=B-7w.$
        \item $\phi(x,y)=(r(x), r'(x)y).$ 
    \end{enumerate}
    In this situation, we clearly have that $\phi^\ast\left(\frac{dX}{Y}\right)=\frac{dx}{y}.$ Our claim reduces to showing that
    $$
    (\#G - 1)\frac{xdx}{y} + \left(-\sum\limits_{Q\in G\setminus\{O\}} x_Q-\frac{t_Q}{x-x_Q}-\frac{u_Q}{(x-x_Q)^2}\right)\frac{dx}{y}
    $$
    is an exact form. Using the fact that
    $$
    y^2-y_Q^2 = x^2+x_Qx+x_Q^2,
    $$
    a straightforward computation shows that
    $$
    d\left(\sum\limits_{Q\in G\setminus\{O\}}\frac{2y}{x-x_Q}\right) =     (\#G - 1)\frac{xdx}{y} + \left(-\sum\limits_{Q\in G\setminus\{O\}} x_Q-\frac{t_Q}{x-x_Q}-\frac{u_Q}{(x-x_Q)^2}\right)\frac{dx}{y},
    $$
    and our claim holds.
\end{proof}

\begin{lemma}\label{lem:scalarActionDeRham}
    Proposition~\ref{prop:IsogenyActionDeRham} is true when $\phi$ is multiplication-by-$m$ with $m\in\Z.$
\end{lemma}

\begin{proof}
    First, if $Q\in E(K),$ we denote by $\tau_Q$ the translation-by-$Q$ on $E.$ A straightforward computation shows that
    $$
    \tau_Q^\ast\left(\frac{xdx}{y}\right) = \frac{xdx}{y}+ \frac{-2yy_Q}{(x_Q-x)^2}\frac{dx}{y}+d\left(\frac{-2y}{x-x_Q}\right)
    $$
    and therefore, 
    $$
    \tau_Q^\ast\left(\frac{xdx}{y}\right) - \frac{xdx}{y}
    $$
    is an exact differential form. The proof of Theorem 5.2 of \cite{Silverman} shows that
    $$
    [m]^\ast\left(\frac{xdx}{y}\right) = m\frac{xdx}{y} + \text{ exact form}.
    $$
    Therefore, it remains to show that 
    \begin{equation}\label{kerSum}
        \sum\limits_{Q\in E[m]\setminus\{O\}} x(Q) = 0.
    \end{equation}
    For $m=0,-1,1,2,$ this is clear. Furthermore, since $[-m]=[-1]\circ[m]$ and $x(-P)=x(P),$ it suffices to show that (\ref{kerSum}) holds for $m\geq 3.$

    If $m$ is odd, then the $x(Q)'s$ are precisely the roots of the $m$-th division polynomial $\psi_m.$ If $m$ is even, then the $x(Q)'s$ are precisely the roots of $\frac{1}{y}\psi_m.$ Therefore, our claim is equivalent to showing that
    $$
    [x^{\deg\psi_{2n+1}-1}]\psi_{2n+1}=0
    $$
    and
    $$
    [x^{\deg(\frac{\psi_{2n}}{y})-1}]\frac{\psi_{2n}}{y}=0,
    $$
    where $[x^l]f$ denotes the coefficient of $x^l$ in $f\in K[x].$
    Now, define
    $$
    \varphi_m(x) = x^{\deg\psi_m}\psi_m\left(\frac{1}{x}\right).
    $$
    Our claim is then equivalent to showing that
    $$
    \frac{d}{dx}\left(\frac{\phi_m}{y^{\delta}}\right)\Bigg|_{x=0}=0,
    $$
    where $\delta = 1$ if $m$ is even and $0$ if $m$ is odd. For $m=0,1,2,3,4,$ this is a direct computation. The claim then follows from induction and the recurrence relations of division polynomials.
\end{proof}

\begin{proof}[Proof of Proposition~\ref{prop:IsogenyActionDeRham}]
    This follows directly from Lemmas~\ref{lem:IsogenyDecomposition} through \ref{lem:scalarActionDeRham}.
\end{proof}

We now provide an algorithm for computing the eigenbasis of a CM action on the de Rham cohomology $H^1_{dR}(E)$ of an elliptic curve $E.$

\begin{algorithm}\label{algorithm:H1dRDiag}
Given an elliptic curve $E/K$ which admits complex multiplication, the following algorithm gives the eigenbasis under complex multiplication.  \begin{enumerate}
    \item Write the elliptic curve $E$ in Weierstrass form while keeping track of the isomorphisms.
    \item Determine the CM order $\O$ of $E$ (for example, see \cite{CremonaSutherland}).
    \item Choose a rational prime $p$ such that $p=N_{L/\Q}(a+b\pi),$ where $a,b\in\Z, L$ is the CM field, and $\pi$ is a generator of $\O.$
    \item Using V\'elu's algorithm (see \cite{VeluFormulas}), find the $p$-isogeny $\phi$ whose codomain has $j$-invariant equal to $j(E).$ This is an action by an element of $\O.$
    \item Use Proposition~\ref{prop:IsogenyActionDeRham} to compute the effect of the CM action on $\frac{dX}{Y}$ and $\frac{XdX}{Y}.$
    \item Diagonalize with respect to the action of $\phi.$
    \item Since $\End(E)$ is commutative, this basis is diagonal with respect to all CM actions.
    \item Use the isomorphisms in (1) to obtain the basis in original form.
\end{enumerate}
\end{algorithm}

We illustrate with an example.

\begin{example*}\label{ex:c1c2E_2Gamma1(4)}
    Consider the elliptic curve $E_2$ where $E_u$ is as in Theorem~\ref{thm:Gamma1(4)Basis}. $E_2$ has complex multiplication by $\Z[2i]$ and we have that $\frac{dx}{y}$ and $\frac{dx}{y}+4\frac{xdx}{y}$ form an eigenbasis of $H_{dR}^1(E_u/\C)$ for the CM action. Plugging this into (1) and (2) of Theorem~\ref{thm:Gamma1(4)Basis}, we obtain Example~\ref{ex:ex3}. 
\end{example*}

An immediate application of this algorithm is a general framework \cite{Zudilin-2nd} for Ramanujan-type formulas for $1/\pi.$ Consider the Legendre model 
$$
E_u^{\Leg}:\quad y^2=x(1-x)(1-ux).
$$ 
Assume $E_u^{\Leg}$ admits CM and  $c_1\frac{dx}y+c_2\frac{xdx}y$ is an eigenvector of the endomorphism ring with $c_2\neq 0.$ Under a fixed real embedding of $\Q(u)$, if $|u|<1$, $|u/(u-1)|<1,$ and $|4u(1-u)|<1$, then the Legendre relation for periods gives an explicitly computable nonzero algebraic number $\alpha_u$ such that
$$\sum_{k\ge 0}(1+ak)\frac{(\frac12)_k^3}{k!^3}\l^k=\frac{\alpha_u}{\pi} \quad \text{where}  \quad a=\frac{-c_2(1+u)}{u}\quad{\text{and}}\quad \l=4u(1-u).$$ For example, when $u=17-12\sqrt{2}$, we have $c_1=1$ and $c_2=-2\sqrt2-3$ which gives $\l=4u(1-u)=-\frac 18$ and $a=6$.  
In this case $\alpha_u=2\sqrt 2$. Note that the real embedding of $c_2$ should be consistent with the choice of embedding of $\Q(u)$ above.

\subsection{B\"onisch, Duhr, and Maggio's Conjectures}\label{subsec:Bonsich}

In Appendix B of \cite{Bonisch}, the authors present several candidates for magnetic modular forms. Although we are not able to prove our results for small primes, we tackle the example given in Page 2 of the introduction. 

In the notation of Page 2, we have that $C_4(\tau)$ defined by \eqref{eq:C4} is modular on $\Gamma_1(2),$ vanishes at the cusps, and has a pole for $\tau$ such that the Hauptmodul $A(\tau)=-\frac{1}{64}.$ On the other hand, we have that (for example, see Proposition 7.1 of \cite{HMM1})
$$
-64A(\tau)=1-\left(\frac{1+\lambda(\tau)}{1-\lambda(\tau)}\right)^2
$$
and therefore, $C_4(\tau)$ can be viewed as a $\Gamma_1(4)\backslash\Gamma_1(2)$-invariant modular form on $\Gamma_1(4)$ with pole of order $2$ at the $\lambda=-1.$ Moreover, $A=\frac{-1}{64}$ is an elliptic point of order $2$ on $\Gamma_1(2)\backslash\mathfrak{H}.$ By 1.3 of \cite{Sch88},  Lemma~\ref{lem:ladicChar}, and Theorem 1 of \cite{HLLT}, we have that for $p\geq 3,$ the characteristic polynomial of Frobenius is given by
$$
P_p(T) = T - \left(\frac{-1}{p}\right)p^2.
$$
Therefore, the proof of Theorem~\ref{thm:main-ASD} implies that for all $m,s\geq 1,$ we have
$$
p^3\cdot c(mp^{s-1})-p^2\left(\frac{-1}{p}\right)c(mp^s)\equiv 0\pmod{p^{3s+2}}.
$$
Dividing by $p^2$ on both sides, we have
$$
c(mp^s)\equiv p\cdot\left(\frac{-1}{p}\right)c(mp^{s-1})\pmod{p^{3s}}.
$$
By induction, this implies that $\frac{c(n)}{n}$ has globally bounded denominators in $\Z\left[\frac{1}{ 2}\right].$  As a side remark, note that $C_4(2\tau)$ is a scalar multiple of $g(\tau)=\frac{-2}{2\pi i}\frac{\l\theta^4 \l'}{(\l-2)^2}$ in Theorem~\ref {thm:Gamma1(4)Basis}  for weight 4 case with $u=2$. 
{The rest of the examples in Appendix B of \cite{Bonisch} can be checked similarly.

\subsection{Zhang's Conjectures}\label{subsec:PengchengConj}

In \cite{Pengcheng}, Zhang puts forth multiple conjectures for meromorphic modular forms of level $1.$ Moreover, using hypergeometric functions and the Borcherds--Shimura lift \cite{Borcherds}, Zhang proves multiple results on magnetic modular forms. Here we discuss which of Zhang's conjectures and theorems follow from our results as well as those which do not. 

Theorem 1.2 of \cite{Pengcheng} follows from Theorem~\ref{thm:main-ASD} by the same argument as in Subsection~\ref{subsec:Bonsich}. For $p\geq k-1,$ Conjecture 2.1 (which Zhang proves for $p\geq 5$) follows from the proof of Theorem~\ref{thm:main-ASD}, again as in Subsection~\ref{subsec:Bonsich}. The authors believe that for $3\leq p\leq k-1,$ the conjecture follows from replacing $\Omega^\bullet_{\mathfrak{u}}$ with another chain complex (see 2.8 (ii) of \cite{Scholl2}). The authors did not verify this.

We note that for the case of supersingular primes $p$ with $p\geq k-1,$ Conjectures 1.4 and 2.3 of \cite{Pengcheng} follow from Lemma~\ref{lem:indDecomp} with $K$ chosen to be the CM field. In this situation, the $F_{k,C}$ are eigenvectors of $F^2$, and the statement follows from noting that the eigenvalues of Frobenius $F^2$ for the elliptic curve are both $p.$ 

For primes $p\geq k-2,$ Zhang's Conjectures 4.11 and 4.13 follow from our Corollary~\ref{cor:ASD-basis}. Note that when an elliptic curve $E$ is defined over $\Q$ and has CM by $\Q(\sqrt{-D})$  where $-D$ is a discriminant with class number, then $E$ is defined over $\Q(\sqrt{-D})$ so we can choose $r=2$. Moreover, the modular forms in Corollary~\ref{cor:ASD-basis} can be taken with coefficients in $\Q.$

Finally, Zhang's Conjecture 4.4 and Theorem 4.7 appear to be stronger than our results as the small primes are included.

\subsection{\texorpdfstring{The noncongruence subgroup $\Gamma_2$}{The noncongruence subgroup Gamma2}}\label{subsec:Gamma_2}

Here we discuss some relevant information on the noncongruence subgroup $\Gamma_2$ and the corresponding example in the introduction.

The group $\Gamma_1(5)$ has $4$ cusps and has genus $0.$  The modular form $t(\tau)=q\prod_{n\ge 1}(1-q^n)^{5\left(\frac n5\right)}$ is a Hauptmodul of $\Gamma_1(5)$ which has a pole at the cusp $0$ and vanishes at the cusp $i\infty.$ The corresponding universal elliptic curve for $\Gamma_1(5)$ is given by 
$$
E_{t}:\quad y^2+(1-t)xy-t y=x^3-t x^2.
$$
Now, consider the modular form defined by the $q$-series expansion
$$
F(\tau):= \frac{1}{2\pi i}\cdot \frac{dt(\tau)}{d\tau}\cdot \sum_{n\ge 0} a_n t(\tau)^n,
$$
where
$\displaystyle a_n=\sum_{k=0}^n \binom{n}{k}^2\binom{n+k}{k}$ are the Ap\'ery numbers. It turns out that $F(\tau)$ is a holomorphic modular form of weight $3$ on $\Gamma_1(5)$ which vanishes at all cusps except at $i\infty.$ This modular form plays an important role in Beukers' proof that $\zeta(2)$ is irrational (see \cite{Beukers-accessory, Zagier-integral}). 

In \cite{ALL}, Atkin, Li, and the second author consider a subgroup $\Gamma_2$ of index $2$ in $\Gamma_1(5)$\footnote{More precisely, they consider the conjugate subgroup $\Gamma^1(5),$ but we use $\Gamma_1(5)$ to simplify notation for our purposes.} whose associated modular curve has genus $0.$ The associated Hauptmodul is given by $t_2=t(\tau)^{\frac{1}{2}}$ and the space of cusp forms of weight $3$ is generated by 
$$
h_2(\tau) = F(\tau)\cdot t_2(\tau)=\sum _{n\ge 1} c(n)q^{n/2}\in\Z\left[\frac{1}{10}\right][[q^{1/2}]].
$$
In their work on making the Atkin and Swinnerton--Dyer congruences obtained in \cite{Scholl2} explicit for this group, the authors compute the associated $\ell$-adic representation and show that for prime $p\geq 7$ and $m,s\ge 1$
$$
c(mp^s)-B_p c(mp^{s-1})+\chi(p) p^2c(mp^{s-2})\equiv 0\pmod {p^{2s}},
$$
where $\gamma_p=1$ and $\chi(\cdot)$ is the Legendre character $\left(\frac{-1}{\cdot}\right)$ and where
$$
\eta(4\tau)^6 = q\prod\limits_{n=1}^\infty(1-q^{4n})^{6}=:\sum\limits_{n\geq 1}B_nq^n.
$$
More precisely, the authors show that the characteristic polynomial in Remark~\ref{rem:computing-P-polynomial} has
$$
\Tr(F_p|\mathcal{W}_{k-2}(\overline{\F_p})) = B_p\ \ \ \text{ and }\ \ \ \det(F_p|\mathcal{W}_{k-2}(\overline{\F_p}))=p^2\cdot\chi(p)
$$
and that for $\iota:u\hookrightarrow X,$ we have
$$
\Tr(F_p|\iota_\ast\iota^\ast\mathcal{G}_{k-2})=p+1-\#E_u(\F_p).
$$

\bibliography{ref}
\bibliographystyle{amsalpha}

\end{document}